\documentclass[11pt,a4paper]{article}
\usepackage{amsmath,amssymb,theorem}
\usepackage{enumitem}

      \usepackage{hyperref}

\usepackage{fullpage}

\newtheorem{thm}{Theorem}[section]
\newtheorem{lemma}[thm]{Lemma}
\newtheorem{prop}[thm]{Proposition}
\newtheorem{cor}[thm]{Corollary}

{\theorembodyfont{\rmfamily}

\newtheorem{example}[thm]{Example}

\newtheorem{rmk}[thm]{Remark}
}

\newcommand{\qed}{\hfill \mbox{\raggedright \rule{.07in}{.1in}}}

\newenvironment{proof}{\vspace{1ex}\noindent{\bf
Proof}\hspace{0.5em}}{\hfill\qed\vspace{1ex}}
\newenvironment{pfof}[1]{\vspace{1ex}\noindent{\bf Proof of
#1}\hspace{0.5em}}{\hfill\qed\vspace{1ex}}

\numberwithin{equation}{section}

\newcommand{\E}{\mathbb{E}}
\newcommand{\N}{\mathbb{N}}
\newcommand{\R}{\mathbb{R}}

\newcommand{\Z}{\mathbb{Z}}

\newcommand{\BBM}{\mathbb{M}}
\newcommand{\BBS}{\mathbb{S}}
\newcommand{\BBQ}{\mathbb{Q}}
\newcommand{\BBW}{\mathbb{W}}

\newcommand{\BBY}{\mathbb{Y}}

\newcommand{\tS}{\tilde{S}}
\newcommand{\tBBS}{\widetilde{\mathbb{S}}}
\newcommand{\tW}{\widetilde{W}}
\newcommand{\tBBW}{\widetilde{\mathbb{W}}}

\newcommand{\Lip}{{\operatorname{Lip}}}
\newcommand{\Leb}{{\operatorname{Leb}}}

\newcommand{\diam}{\operatorname{diam}}

\newcommand{\cB}{{\mathcal{B}}}

\newcommand{\cG}{{\mathcal{G}}}
\newcommand{\cM}{{\mathcal{M}}}

\newcommand{\bbeta}{{\beta}}
\newcommand{\TT}{{t_1}}

\newcommand{\SMALL}{\textstyle}
\newcommand{\BIG}{\displaystyle}

\begin{document}

\title{Deterministic homogenization under optimal moment assumptions for fast-slow systems.  Part 1.}

\author
{A. Korepanov\thanks{Mathematics Department, University of Exeter, Exeter, EX4 4QF, UK}
\and
Z. Kosloff\thanks{Einstein Institute of Mathematics,
Hebrew University of Jerusalem,
Givat Ram. Jerusalem, 9190401, ISRAEL}
\and
  I. Melbourne\thanks{Mathematics Institute, University of Warwick, Coventry, CV4 7AL, UK}
}

\date{14 June 2020}

 \maketitle

\begin{abstract}
We consider deterministic homogenization (convergence to a stochastic differential equation) for multiscale systems of the form
\[
 x_{k+1} = x_k + n^{-1} a_n(x_k,y_k) + n^{-1/2} b_n(x_k,y_k), \quad y_{k+1} = T_n y_k,
\]
where the fast dynamics is given by a family $T_n$ of nonuniformly expanding maps.   
Part~1 builds on our recent work on martingale approximations for families of nonuniformly expanding maps.
We prove an iterated weak invariance principle and establish optimal iterated moment bounds for such maps.  (The iterated moment bounds are new even for a fixed nonuniformly expanding map $T$.)
The homogenization results are a consequence of this together with parallel developments on rough path theory in Part 2 by Chevyrev, Friz, Korepanov, Melbourne \& Zhang.
\end{abstract}

\section{Introduction}

Recently, there has been a great deal of interest in deterministic homogenization~\cite{CFKMZ19,SimoiLiverani16,SimoiLiveranisub,Dolgopyat04, Dolgopyat05, GM13b, KM16,KM17,KKM18,MS11} whereby deterministic multiscale systems converge to a stochastic differential equation as the time-scale separation goes to infinity.  A byproduct of this is a deeper understanding~\cite{KM16} of the correct interpretation of limiting stochastic integrals~\cite{WongZakai65}.

Using rough path theory~\cite{FrizHairer,Lyons98}, it was shown in~\cite{KM16,KM17} that
homogenization reduces to proving 
certain statistical properties
for the fast dynamics.  
These statistical properties take the form of an ``iterated invariance principle'' (iterated WIP) 
which gives the correct interpretation of the limiting stochastic integrals, 
and control of ``iterated moments''
which provides tightness in the rough path topology used for proving convergence.
In particular, the homogenization question was settled in~\cite{KM16,KM17} for uniformly expanding/hyperbolic fast (Axiom~A) dynamics and for nonuniformly expanding/hyperbolic fast dynamics modelled by Young towers with exponential tails~\cite{Young98}.
The results in~\cite{KM16,KM17} also covered fast dynamics modelled by Young towers with polynomial tails~\cite{Young99} but the results were far from optimal.  
It turns out that advances on two separate fronts are required to obtain optimal results: 
\begin{itemize}
\item[(i)] Martingale methods for nonuniformly expanding maps modelled by Young towers, yielding optimal control of iterated moments; 
\item[(ii)] Discrete-time rough path theory in  $p$-variation topologies, relaxing the required control for ordinary and iterated moments.
\end{itemize}
These two directions rely on techniques in smooth ergodic theory and in stochastic analysis respectively, so the homogenization question divides naturally into two parts.
This paper {\em Part 1} covers the ergodic-theoretical aspects required for (i), while the rough path aspects required for (ii) are dealt with in {\em Part 2} by Chevyrev~{\em et al.}~\cite{CFKMZsub}.
As we explain below, together these provide an optimal solution to the
homogenization question when the fast dynamics is given by a nonuniformly expanding map or a family of such maps.
\\

The homogenization question that we are interested in takes the following form.
Let $T_n:\Lambda\to \Lambda$, $n\ge1$, be a family of dynamical systems with
ergodic invariant probability measures $\mu_n$.
Consider the fast-slow system
\begin{align} \label{eq:fs}
 x_{k+1} = x_k + n^{-1} a_n(x_k,y_k) + n^{-1/2} b_n(x_k,y_k), \quad y_{k+1} = T_n y_k,
\end{align}
where $x_k=x_k^{(n)}$ takes values in $\R^d$ with $x_0\equiv \xi\in\R^d$, and $y_k$ takes values in $\Lambda$.
Our main assumption is that $T_n$ is a uniform family of nonuniformly expanding maps of order $p>2$ as in~\cite{KKM18} (see Section~\ref{sec:NUEn}  below for precise definitions).
We impose mild regularity conditions on $a_n,\,b_n:\R^d\times \Lambda\to \R^d$ and require that
$\int_\Lambda b_n(x,y)\,d\mu_n(y)=0$ for all $x\in\R^d$,
$n\ge1$.

Define $\hat x_n(t)=x^{(n)}_{[nt]}$ and let $\lambda_n$ be a family of probability measures on $\Lambda$.
We regard $\hat x_n$ as a sequence of random variables on the probability spaces $(\Lambda,\lambda_n)$ with values 
in the Skorohod space $D([0,1],\R^d)$.
The aim is to prove weak convergence, 
\(
\hat x_n\to_{\lambda_n} X\;
\text{ as $n\to\infty$},
\)
where $X$ is the solution to a
stochastic differential equation.

\begin{example}
\label{ex:LSV}
To fix ideas, we focus first on the case where $T_n\equiv T$ is a single nonuniformly expanding map.
Pomeau-Manneville intermittent maps~\cite{PomeauManneville80} provide the prototypical examples of such maps.  We consider in particular the class of intermittent maps studied in~\cite{LSV99}, namely
\begin{align} \label{eq:LSV}
T:[0,1]\to[0,1], \qquad
Tx=\begin{cases}
x(1+2^\gamma x^\gamma) & x<\frac12 \\
2x-1 & x>\frac12
\end{cases}.
\end{align}
Here $\gamma>0$ is a parameter and there is a unique absolutely continuous invariant probability measure $\mu$ provided $\gamma<1$.
Moreover, the central limit theorem (CLT) holds for all H\"older observables $v:[0,1]\to\R$, provided $\gamma<\frac12$.
By~\cite{Gouezel04}, the CLT fails  for typical H\"older observables once $\gamma>\frac12$. Even for $\gamma=\frac12$, the CLT requires a nonstandard normalization.  Hence it is natural to restrict here to the range $\gamma\in(0,\frac12)$.
(The range $\gamma\in(\frac12,1)$ leads to superdiffusive phenomena~\cite{Gouezel04,MZ15} and we refer to~\cite{CFKMsub,GM13b} for the homogenization theory for the corresponding fast-slow systems.)

The homogenization problem for fast-slow systems driven by such intermittent maps $T$ (with $\lambda_n\equiv\mu$) was previously considered in~\cite{KM16} and then~\cite{CFKMZ19}.  The techniques therein sufficed  in the restricted range $\gamma\in(0,\frac{2}{5})$ and even then only in the special case $b(x,y)=h(x)v(y)$ where
$h:\R^d\to\R^{d\times m}$, $v:\Lambda\to\R^d$.  There are two additional steps, covered in Parts~1 and~2 respectively, that lead to homogenization in the full range $\gamma\in(0,\frac12)$ and for general $b$:
\begin{itemize}
\item[(i)] As mentioned above, to obtain homogenization results it suffices to prove the
iterated WIP and control of iterated moments.  These statistical properties are formulated at the level of the map $T$ for H\"older observables $v:[0,1]\to\R^d$ with $\int v\,d\mu=0$.
The iterated WIP was already proved in~\cite{KM16} in the full range $\gamma\in(0,\frac12)$. 
Define 
\[
S_nv=\sum_{0\le j<n} v \circ T^j, \qquad
\BBS_nv=\sum_{0\le i<j<n} (v \circ T^i) \otimes (v \circ T^j).
\]
There are numerous methods for estimating ordinary moments $|S_nv|_{2(p-1)}$
for $p<1/\gamma$.
Estimates for iterated moments 
    $|\BBS_nv|_{2(p-1)/3}$ 
were given in~\cite{KM16}.
In Theorem~\ref{thm:moments} of the current paper, we estimate 
    $|\BBS_nv|_{p-1}$;
this is the first result giving optimal estimates for iterated moments.
Using~\cite{CFKMZ19},
we can then cover the full range $\gamma\in(0,\frac12)$ in the product case $b(x,y)=h(x)v(y)$.
\item[(ii)] 
The papers~\cite{KM16,KM17} use rough path theory in H\"older spaces.  
However,
H\"older rough path theory requires control of the ordinary moments $|S_nv|_{2q}$ and
the iterated moments $|\BBS_nv|_q$ for some $q>3$.
As shown in~\cite[Section~3]{MTorok12},
such control even for the ordinary moments requires $\gamma<\frac14$.
The papers by Chevyrev~{\em et al.}~\cite{CFKMZ19,CFKMZsub}
employ rough path theory in $p$-variation spaces and require iterated moment estimates only for $q>1$.  
Whereas~\cite{CFKMZ19} is restricted to the product case $b(x,y)=h(x)v(y)$,
Part~2~\cite{CFKMZsub} covers general $b$ following~\cite{KM17}.
This method combined with the previous iterated WIP and iterated moment estimates in~\cite{KM16} covers the range $\gamma\in(0,\frac25)$
for general $b$.
\end{itemize}
Combining (i) and (ii), we cover the optimal range $\gamma\in(0,\frac12)$
for general $b$.

In addition, we obtain the homogenization result $\hat x_n\to_{\lambda_n}X$ for a larger class of measures 
including the natural choice $\lambda_n\equiv\Leb$.

Returning to families of nonuniformly expanding maps, in~\cite{KKM18} we
considered intermittent maps $T_n:[0,1]\to[0,1]$, $n\in\N\cup\{\infty\}$, as in~\eqref{eq:LSV}
with parameters $\gamma_n$ 
such that $\lim_{n\to\infty}\gamma_n=\gamma_\infty$.
Homogenization results  with $\lambda_n=\mu_n$ and $\lambda_n\equiv\mu_\infty$ were obtained in~\cite{KKM18} for a restricted class of fast-slow systems with 
$b_n(x,y)=h_n(x)v_n(y)$, $h_n$ exact, for $\gamma_\infty\in(0, \frac{1}{2})$.  (For such systems, rough path theory was not needed.)
By the results in this paper, combined with those in Part~2,
we treat general $b_n$, again in the full
range $\gamma_\infty\in(0,\frac12)$.  Moreover, we cover a larger class of measures including 
$\lambda_n\equiv\Leb$.  
\end{example}

 The remainder of Part~1 is organized as follows.
 In Sections~\ref{sec:NUE} and~\ref{sec:NUEn}, we consider nonuniformly expanding maps (fixed, and in uniform families~\cite{KKM18}, respectively).
 In particular, we obtain optimal estimates for iterated moments in Theorem~\ref{thm:moments} and the iterated WIP for families in Theorem~\ref{thm:WIP}.
In Section~\ref{sec:ex}, we consider examples including the intermittent maps in
Example~\ref{ex:LSV}.
The theory is extended to families of nonuniformly expanding semiflows in
Section~\ref{sec:flow}.

We refer to Part~2 for the parallel developments in rough path theory and a precise statement and proof of homogenization for the fast-slow systems~\eqref{eq:fs}.

\vspace{-2ex}
\paragraph{Notation}
%
For $a,b\in\R^d$, we define the outer product
$a\otimes b=ab^T\in\R^{d\times d}$. 
For $J\in \R^{m\times n}$, we use the norm
$|J|=\big(\sum_{i=1}^m\sum_{j=1}^n J_{ij}^2\big)^{1/2}$.
Then \(|a \otimes b| \leq |a| |b|\) for $a,b\in\R^d$.
 
 For real-valued functions $f,\,g$, the integral $\int f\,dg$ denotes
 the It\^o integral (where defined).  
 Similarly, for vector-valued functions,
 $\int f\otimes dg$ denotes matrices of It\^o integrals.

We use ``big O'' and $\ll$ notation interchangeably, writing $a_n=O(b_n)$ or $a_n\ll b_n$
if there are constants $C>0$, $n_0\ge1$ such that
$a_n\le Cb_n$ for all $n\ge n_0$.
As usual, $a_n=o(b_n)$ means that $\lim_{n\to\infty}a_n/b_n=0$.

Recall that $v:\Lambda\to\R$ is a H\"older observable on a metric space $(\Lambda,d_\Lambda)$ if
  $\|v\|_{\eta} = |v|_\infty + |v|_\eta<\infty$ where $|v|_\infty=\sup_\Lambda|v|$,
$|v|_\eta= \sup_{x\neq y} |v(x)-v(y)|/d_\Lambda(x,y)^\eta$.

\section{Nonuniformly expanding maps}
\label{sec:NUE}

In this section, we recall and extend the results in~\cite{KKM18} for nonuniformly expanding maps.

Let $(\Lambda,d_\Lambda)$ be a bounded metric space with finite Borel measure $\rho$ and let $T:\Lambda\to\Lambda$ be a nonsingular transformation.
Let $Y\subset \Lambda$ be a subset of positive measure, and
let $\alpha$ be an at most countable measurable partition
of $Y$.    We suppose that there is an integrable
{\em return time} function $\tau:Y\to\Z^+$, constant on each $a$ with
value $\tau(a)\ge1$, and constants $\bbeta>1$, $\eta\in(0,1]$, $C_1\ge1$
such that for each $a\in\alpha$,
\begin{itemize}
\item[(1)] $F=T^\tau$ restricts to a (measure-theoretic) bijection from $a$  onto $Y$.
\item[(2)] $d_\Lambda(Fx,Fy)\ge \bbeta d_\Lambda(x,y)$ for all $x,y\in a$.
\item[(3)] $d_\Lambda(T^\ell x,T^\ell y)\le C_1d_\Lambda(Fx,Fy)$ for all $x,y\in a$,
$0\le \ell <\tau(a)$.
\item[(4)] $\zeta_0=\frac{d\rho|_Y}{d\rho|_Y\circ F}$
satisfies $|\log \zeta_0(x)-\log \zeta_0(y)|\le C_1d_\Lambda(Fx,Fy)^\eta$ for all
\mbox{$x,y\in a$}.
\end{itemize}

Such a dynamical system $T:\Lambda\to \Lambda$ is called {\em nonuniformly expanding}.
(It is not required that $\tau$ is the first return time to $Y$.)
We refer to the induced map $F=T^\tau:Y\to Y$ as a {\em uniformly expanding map}.
There is a unique absolutely continuous $F$-invariant probability measure $\mu_Y$
on $Y$ and $d\mu_Y/d\rho\in L^\infty$.

Define the (one-sided) Young tower map~\cite{Young99}, $f_\Delta:\Delta\to\Delta$, 
\[
\Delta=\{(y,\ell)\in Y\times\Z:0\le\ell\le \tau(y)-1\},
\qquad
f_\Delta(y,\ell)=\begin{cases} (y,\ell+1), & \ell\le \tau(y)-2
\\ (Fy,0), & \ell=\tau(y)-1 \end{cases}.
\]
The projection
$\pi_\Delta:\Delta\to \Lambda$, $\pi_\Delta(y,\ell)=T^\ell y$, defines a semiconjugacy
from $f_\Delta$ to $T$.  Define the ergodic $f_\Delta$-invariant
probability measure
$\mu_\Delta=\mu_Y\times\{{\rm counting}\}/\int_Y \tau\,d\mu_Y$ on $\Delta$.
Then
$\mu=(\pi_\Delta)_*\mu_\Delta$ is an absolutely continuous ergodic  $T$-invariant probability measure on $\Lambda$.

In this section, we work with a fixed
nonuniformly expanding map $T:\Lambda\to \Lambda$ with
induced map $F=T^\tau:Y\to Y$ where $\tau\in L^p(Y)$ for some $p\ge2$,\footnote{
In~\cite{KKM18}, we considered this set up with $p\ge1$.  Since we have no new results for $p<2$ beyond those already in~\cite{KKM18}, we restrict in this paper to the case $p\ge2$.
}
and Young tower map $f_\Delta:\Delta\to\Delta$.
The corresponding ergodic invariant probability measures are denoted $\mu$, $\mu_Y$ and $\mu_\Delta$.
Throughout, \(|\;|_p\) denotes the $L^p$-norm on
\((\Lambda,\mu)\), 
\((Y,\mu_Y)\) and 
$(\Delta,\mu_\Delta)$ as appropriate.
Also, $\|\;\|_\eta$ denotes the H\"older norm on $\Lambda$ and $Y$.

Although the map $T$ is fixed, the dependence of various constants on $T$ is important in later sections.
To simplify the statement of results in this section, we denote by \(C\) various constants
depending continuously on $\diam\Lambda$, $C_1$, $\bbeta$, $\eta$, $p$ and 
$|\tau|_p$.

\vspace{1ex}

Let $L:L^1(\Delta)\to L^1(\Delta)$
and $P:L^1(Y)\to L^1(Y)$ denote the transfer 
operators corresponding to $f_\Delta: \Delta \to \Delta$ and
$F: Y \to Y$.
(So 
\(
  \int_{\Delta} Lv\,w\,d\mu_\Delta
  = \int_\Delta v\,w\circ f_\Delta\,d\mu_\Delta
\)
for $v\in L^1(\Delta)$, $w\in L^\infty(\Delta)$, and 
\(
  \int_Y Pv\,w\,d\mu_Y
  = \int_Y v\,w\circ F\,d\mu_Y
\)
for $v\in L^1(Y)$, $w\in L^\infty(Y)$.)

Let  \(\zeta = d \mu_Y/d\mu_Y \circ F\).
Given $y\in Y$ and \(a \in \alpha\), let
$y_a$ denote the unique $y_a\in a$ with $F y_a=y$.
Then we have the pointwise expression for $L$,
\begin{equation}
  \label{eq:PL}
  (L v)(y, \ell) = 
  \begin{cases}
    \sum_{a \in \alpha} \zeta(y_a) v(y_a, \tau(y_a)-1), & \ell=0 \\
    v(y, \ell-1), & 1\le \ell \le \tau(y)-1
  \end{cases}.
\end{equation}

\subsection{Martingale-coboundary decomposition}
\label{sec:mart}

Let $T:\Lambda\to\Lambda$ be a nonuniformly expanding map as above with return time $\tau\in L^p(Y)$, $p\ge2$.
Fix $d\ge1$ and let $v\in C^\eta(\Lambda,\R^d)$ with $\int_\Lambda v\,d\mu=0$.
 Define the lifted observable $\phi=v\circ\pi_\Delta : \Delta \to \R^d$.

We recall the martingale-coboundary decomposition
\begin{equation}
  \label{eq:mcd}
  \phi = m + \chi \circ f_\Delta - \chi
  , \qquad m \in \ker L
\end{equation}
from \cite[Section~2.2]{KKM18}, which is obtained as follows.
First, define the \emph{induced observable} $\phi': Y \to \R^d$ by
\(\phi'(y) = \sum_{\ell=0}^{\tau(y)-1} \phi(y,\ell)\).
Next, define \(\chi', m' : Y \to \R^d\) by
\(
  \chi'=\sum_{k=1}^\infty P^k\phi'
\)
and
\(
  \phi'=m'+\chi' \circ F-\chi'
\). Let
\begin{equation} \label{eq:chi}
  \chi(y, \ell) = 
  \chi'(y)+ \sum_{k=0}^{\ell-1} \phi(y, k)
  \quad \text{and} \quad
  m(y, \ell) = \begin{cases}
    \hphantom{Y} 0, & \ell \leq \tau(y)-2 \\
    m'(y), & \ell=\tau(y)-1
  \end{cases}.
\end{equation}

By \cite[Section~2.2]{KKM18}, 
\(\|\chi'\|_\eta \leq C \|v\|_\eta\).
Furthermore,
\begin{prop}
  \label{prop:m}
  \(|m|_p \leq C \|v\|_\eta\), \(|\chi|_{p-1} \leq C \|v\|_\eta\)
  and for all \(n \geq 1\), \(q \geq p\),
  \[
    \big| \max_{k \leq n} | \chi \circ f_\Delta^k - \chi| \big|_p
    \leq C \|v\|_\eta \big(n^{1/q} + n^{1/p} |1_{\{\tau \geq n^{1/q}\}} \tau|_p\big)
    .
  \]
(In particular, $\big|\max_{k \leq n} | \chi \circ f_\Delta^k - \chi| \big|_p\le C'\|v\|_\eta n^{1/p}$.)
\end{prop}

\begin{proof}
  See \cite[Propositions 2.4 and 2.7]{KKM18}.
\end{proof}

\begin{prop}
  \label{prop:3}
  \(\big|L^n|m|^p\big|_\infty \leq C \|v\|_\eta^p\) for all \(n \geq 1\).
\end{prop} 

\begin{proof}
  Using~\eqref{eq:PL} and the definition of $m$, we have
  \[
    (L|m|^p)(y,\ell)
    = \begin{cases}
\sum_{a \in \alpha} \zeta(y_a) |m'(y_a)|^p
      , & \ell=0\\
      \hphantom{YYY}0, & 1\leq \ell\leq \tau(y)-1
    \end{cases}.
  \]
  Note that $|m'|\le 2|\chi'|_\infty+|\phi'| 
  \le 2|\chi'|_\infty+\tau|v|_\infty
\ll \tau\|v\|_\eta$. 
Also 
  $|1_a\zeta|_\infty \ll  \mu_Y(a)$ 
(see for example~\cite[Proposition~2.2]{KKM18}).
Hence
  \[
    L|m|^p 
    \ll \SMALL \sum_{a\in\alpha}\mu_Y(a)\tau(a)^p\|v\|_\eta^p
    = |\tau|_p^p\|v\|_\eta^p
    \ll \|v\|_\eta^p
    .
  \]
  Hence,
  \(\big|L^n|m|^p\big|_\infty \leq \big|L|m|^p\big|_\infty \ll \|v\|_\eta^p\)
  for all \(n \geq 1\).
\end{proof}

Let $\breve\phi=UL(m\otimes m)-\int_\Delta m\otimes m \,d\mu_\Delta:\Delta\to\R^{d\times d}$ where
$U$ is the Koopman operator $U\phi=\phi\circ f_\Delta$.

\begin{prop} \label{prop:breve}
$\big|\max_{k\le n}|\sum_{j=0}^{k-1}\breve\phi\circ f_\Delta^j|\big|_p
\le Cn^{1/2}\|v\|_\eta^2$.
\end{prop}

\begin{proof} See~\cite[Corollary~3.2]{KKM18}.
\end{proof}

\subsection{Moment estimates}
\label{sec:moments}

Given $v\in C^\eta(\Lambda,\R^d)$ with $\int_\Lambda v\,d\mu=0$, we define
\begin{align} \label{eq:BBS}
S_nv=\sum_{0\le j<n} v \circ T^j, \qquad
\BBS_nv=\sum_{0\le i<j<n} (v \circ T^i) \otimes (v \circ T^j).
\end{align}
The main result in this section is the estimate for iterated moments 
    $\big|\max_{k\le n}|\BBS_kv|\big|_{p-1}$ in the next theorem.

\begin{thm}[Iterated moments] 
  \label{thm:moments}
  For all \(n \geq 1\), 
  \[
    \big|\max_{k\le n}|S_kv|\big|_{2(p-1)}\le Cn^{1/2}\|v\|_\eta
    , \qquad
    \big|\max_{k\le n}|\BBS_kv|\big|_{p-1}\le Cn \|v\|_\eta^2
    .
  \]
\end{thm}

\begin{proof}
   Since $p\ge2$, 
  the estimate for $S_nv$ is given in~\cite[Corollary~2.10]{KKM18}.
  It remains to prove the bound for $\BBS_nv$, equivalently
$\BBS_n\phi=
\sum_{0\le i<j<n} (\phi \circ f_\Delta^i) \otimes (\phi \circ f_\Delta^j)$.
  Using~\eqref{eq:mcd}, 
  \[
    \BBS_n\phi
    = \sum_{0\le j<n}
      (\chi \circ f_\Delta^j-\chi) \otimes (\phi \circ f_\Delta^j)
      + \sum_{0\le i< j<n} (m \circ f_\Delta^i) \otimes (\phi \circ f_\Delta^j)
    = I_n + J_n.
  \]
  By Proposition~\ref{prop:m},
  \[
    \big|\max_{k\le n}|I_k|\big|_{p-1}
\le |\phi|_\infty \sum_{0\le j<n}|\chi \circ f_\Delta^j-\chi|_{p-1}
    \leq 2n|v|_\infty
    |\chi|_{p-1} 
    \ll n\|v\|_\eta^2
    .
  \]
  Next,
    $J_n = \sum_{i=0}^{n-2} (m\circ f_\Delta^i) \otimes \big(\big(\sum_{j=1}^{n-i-1}\phi\circ f_\Delta^j)\circ f_\Delta^i\big) = \sum_{\ell=2}^nX_{n,\ell}$ where
  \[
    X_{n,\ell} = \Big( m \otimes \sum_{j=1}^{\ell-1} \phi \circ f_\Delta^j \Big) \circ f_\Delta^{n-\ell}
     = \big( m \otimes \{(S_{\ell-1}\phi)\circ f_\Delta\} \big) \circ f_\Delta^{n-\ell}
    .
  \]  
Now, $|S_n\phi|_p\le |S_n\phi|_{2(p-1)}\ll n^{1/2}\|v\|_\eta$ since $p\ge2$.
  Hence by Proposition \ref{prop:3},  
  \begin{align*}
    |X_{n,\ell}|_p^p
   &  \le \int_{\Delta} |m|^p | (S_{\ell-1}\phi)\circ f_\Delta |^p
      \, d\mu_\Delta
     =\int_{\Delta} L|m|^p |S_{\ell-1}\phi|^p
      \, d\mu_\Delta
    \\ & \le \big| L|m|^p\big|_\infty |S_{\ell-1}\phi|_p^p
    \ll \ell^{p/2} \|v\|_\eta^{2p}
    \ll n^{p/2} \|v\|_\eta^{2p},
  \end{align*}
so $|X_{n,\ell}|_p^2\ll n\|v\|_\eta^4$.

Let $\cM$ denote the underlying $\sigma$-algebra on $(\Delta,\mu_\Delta)$ and define
$\cG_{n,\ell}=f_\Delta^{-(n-\ell)}\cM$, $2\le \ell \le n$.  
  Since \(L m = 0\),
  \[
    L^{n+1-\ell} X_{n,\ell} = L\big( m \otimes \{(S_{\ell-1}\phi)\circ f_\Delta\} \big)=Lm \otimes (S_{\ell-1}\phi) =0
  \]
  for all \(\ell\).
It follows (cf.\ \cite[Proposition~2.9]{KKM18}) that
$\{X_{n,\ell},\cG_\ell;\,2\le \ell\le n\}$ is a sequence of martingale differences.
  Working coordinatewise, by Burkholder's inequality~\cite{Burkholder73},
  \[
    \big|\max_{k\le n}|J_k|\big|_p^2
     \ll \Big|\Big(\sum_{\ell=2}^n X_{n,\ell}^2\Big)^{1/2}\Big|_p^2
     = \Big| \sum_{\ell=2}^n  X_{n,\ell}^2\Big|_{p/2}
     \leq \sum_{\ell=2}^n  |X_{n,\ell}^2|_{p/2}
    = \sum_{\ell=2}^n |X_{n,\ell}|_p^2 \ll n^2\|v\|_\eta^4,
  \]
  and so
  \(
    \big| \max_{k\le n}|J_k|\big|_p \ll n \|v\|_\eta^2
    . 
  \) 
  This completes the proof.
\end{proof}

\paragraph{Moments on $\Delta$}

It is standard that the moment estimates for $v:\Lambda\to\R^d$ follow from corresponding estimates for lifted observables $\phi=v\circ\pi_\Delta:\Delta\to\R^d$.
In Proposition~\ref{prop:EeeE}, we need such an estimate for an observable on $\Delta$ that need not be the lift of an observable on $\Lambda$.
Hence, we recall now how to derive moment estimates on $\Delta$.

We define a metric on $\Delta$ based on the metric $d_\Lambda$ on $Y$:
\begin{equation}
    \label{eq:sepapa}
    d_\Delta((y, \ell), (y', \ell'))
    = \begin{cases}
        d_\Lambda(Fy, Fy') & \ell = \ell' \text{ and } y,y' \text{ are in the same } a \in \alpha
        \\
        \diam \Lambda & \text{else}
    \end{cases}
    .
\end{equation}

\begin{rmk}
    In~\eqref{eq:sepapa}, if we use a symbolic metric on Y in place of $d_\Lambda$,
    then $d_\Delta$ is the usual symbolic metric on $\Delta$.
\end{rmk}

As usual, $\| \ \|_\eta$ denotes the H\"older norm on $\Delta$.
From the definition of nonuniformly expanding map,
$d_\Lambda(T^\ell y, T^{\ell'} y') \leq C_1 d_\Delta((y,\ell), (y',\ell'))$; hence
if $v : \Lambda \to \R^d$ is H\"older then so is
its lift $\phi = v \circ \pi : \Delta \to \R^d$.
Moreover, $f_\Delta$ is itself a nonuniformly expanding map on $(\Delta, d_\Delta)$
with the same constants as $T$, so Theorem~\ref{thm:moments} yields:

\begin{lemma} \label{lem:phi}
Let $\phi:\Delta\to\R^d$ with $\|\phi\|_\eta<\infty$, such that $\int_\Delta \phi\,d\mu_\Delta=0$.
Define $S_n\phi=\sum_{0\le j<n} \phi \circ f_\Delta^j$ and
\(
\BBS_n\phi=\sum_{0\le i<j<n} (\phi \circ f_\Delta^i) \otimes (\phi \circ f_\Delta^j).
\)
Then
  \[
    \big|\max_{k\le n}|S_k\phi|\big|_{2(p-1)}\le Cn^{1/2}\|\phi\|_\eta
    \qquad \text{and} \qquad
    \big|\max_{k\le n}|\BBS_k\phi|\big|_{p-1}\le Cn \|\phi\|_\eta^2
    . 
  \]

\vspace{-6ex}
\qed
\end{lemma}

\subsection{Drift and diffusion coefficients}
\label{sec:drift}

Let $S_nv$, $\BBS_nv$ be as in~\eqref{eq:BBS} and define $\Sigma,\,E\in \R^{d\times d}$,
\begin{align} \label{eq:drift}
  \Sigma & = \lim_{n \to \infty} \frac1n
  \int_\Lambda S_nv \otimes S_nv
  \, d\mu
  , 
  \qquad
  E  = \lim_{n\to\infty} \frac1n
  \int_{\Lambda} \BBS_nv \, d\mu
  . 
\end{align}

\begin{prop}
  \label{prop:drift}
  The limits in~\eqref{eq:drift} exist and are given by
  \[
\SMALL
\Sigma = \int_\Delta m \otimes m \, d\mu_\Delta, \qquad
E = \int_\Delta \chi \otimes \phi \, d\mu_\Delta.
\]
Moreover, for all $n\ge1$,
\[
  \Big|\frac1n \int_\Lambda S_nv \otimes S_nv \, d\mu-\Sigma\Big|
\le C\|v\|_\eta^2 n^{1/p-1/2},
\qquad
\Big|\frac1n\int_{\Lambda} \BBS_nv \, d\mu-E\Big|\le C\|v\|_\eta^2(n^{-1/2}+n^{-(p-2)}).
\]
\end{prop}

\begin{proof}
The limit for $\Sigma$ is obtained in  \cite[Corollary~2.12]{KKM18}. 
The proof of 
\cite[Corollary~2.12]{KKM18} contains the estimate
\[
  \Big|\frac1n \int_\Lambda S_nv \otimes S_nv \, d\mu-\Sigma\Big|
\ll n^{-1/2}\|v\|_\eta |\chi\circ f_\Delta^n-\chi|_p,
\]
so the convergence rate for $\Sigma$ follows from Proposition~\ref{prop:m}.
  
Next, we note that 
  \(|\chi\otimes (n^{-1}S_n\phi)|_1 \le |\chi|_1|v|_\infty<\infty\) 
since $\chi\in L^{p-1}\subset L^1$.   Also 
$n^{-1}S_n\phi\to0$
  almost surely by the pointwise ergodic theorem. 
Hence it follows from the dominated convergence theorem that
\[
\lim_{n\to\infty} 
    \frac1n\int_{\Delta}\chi\otimes S_n\phi \, d\mu_\Delta=0.
\]
  Since $Lm=0$, we have $\int_\Delta (m\circ f_\Delta^i)\otimes(\phi\circ f_\Delta^j)\,d\mu_\Delta=0$ 
  for all $i<j$. Hence by~\eqref{eq:mcd},
  \begin{align*}
    \int_{\Delta} 
     \BBS_n\phi\,d\mu_\Delta
     & =
    \int_{\Delta} \sum_{j=1}^{n-1} (\chi\circ f_\Delta^j-\chi) \otimes (\phi\circ f_\Delta^j) \, d\mu_\Delta
     = n\int_\Delta \chi\otimes\phi\,d\mu_\Delta
    - 
    \int_{\Delta}\chi\otimes S_n\phi \,d\mu_\Delta .
  \end{align*}
It follows that $E=\lim_{n\to\infty}\frac1n\int_{\Delta} \BBS_n\phi
\, d\mu_\Delta=\int_\Delta \chi\otimes\phi\,d\mu_\Delta$.

To obtain the convergence rate for $E$, we may suppose without loss that $p\in(2,\frac52]$.
Write $(p-1)^{-1}+q^{-1}=1$ where $q\in[3,\infty)$.
It follows from H\"older's inequality and Proposition~\ref{prop:m} that $|\chi\otimes S_n\phi|_1\le |\chi|_{p-1}|S_n\phi|_q
\ll\|v\|_\eta |S_n\phi|_q$.
By Theorem~\ref{thm:moments},
\[
\int_\Delta |S_n\phi|^q\,d\mu_\Delta
 \le |S_n\phi|_\infty^{q-2(p-1)}
\int_\Delta |S_n\phi|^{2(p-1)}\,d\mu_\Delta
 \ll \|v\|_\eta^q\, n^{q-2(p-1)}n^{p-1}
= \|v\|_\eta^q\, n^{q-(p-1)}.
\]
Hence
$|\chi\otimes S_n\phi|_1\ll 
\|v\|_\eta^2\, n^{1-(p-1)/q}
=\|v\|_\eta^2\, n^{3-p}$
and the result follows.
\end{proof}

For later use, we record the following result:

\begin{prop}
  \label{prop:EeeE}
For $n\ge1$,
  \[
    \biggl|
      \max_{k \leq n} \Big| \sum_{j=0}^{k-1} 
      \big((\chi \otimes \phi) \circ f_\Delta^j - E\big) \Big|
    \biggr|_1
    \leq C \|v\|_\eta^2 \Big(
      n^{3/4} + n \int_Y \tau^2 1_{\{\tau \geq n^{1/4}\}} \, d\mu_Y
    \Big)
    .
  \]
\end{prop}

\begin{proof}
  Fix $q>0$, and define 
  \[
\psi : \Delta \to \R^{d \times d}, \qquad
    \psi(y, \ell) = (\chi\otimes\phi) (y, \ell) 1_{\{\tau(y) \geq q\}}
    .
  \]
  By~\eqref{eq:chi},
\(
|\chi\otimes\phi|(y, \ell)
\le
(|\chi'|_\infty+\ell|v|_\infty)|v|_\infty \ll \|v\|_\eta^2\tau(y).
\)
Hence
\begin{equation} \label{eq:psi1}
    |\psi|_1 
    \ll \|v\|_\eta^2 \int_\Delta \tau(y) 1_{\{\tau(y) \geq q\}} \, d\mu_\Delta(y, \ell)
    \leq \|v\|_\eta^2 \int_Y \tau^2 1_{\{\tau \geq q\}} \, d\mu_Y.
\end{equation}
  
Write $\chi\otimes\phi-E=U+V$ where
  \[
U = \psi - \int_\Delta \psi \, d\mu_\Delta,
\qquad  V = \chi\otimes \phi -\psi  - \int_\Delta(\chi\otimes\phi-\psi) \, d\mu_\Delta.
\]

  By~\eqref{eq:psi1},
\[
    \Big| \max_{k \le n} |S_kU| \Big|_1
    \leq \Big| \sum_{j < n} |U \circ f_\Delta^j| \Big|_1
    \leq n |U|_1 \le 2n|\psi|_1
    \ll n \|v\|_\eta^2 
   \int_Y \tau^2 1_{\{\tau \geq q\}} \, d\mu_Y
    .
\]

Next, $V(y,\ell)=(\chi\otimes \phi)(y,\ell)1_{\{\tau< q\}}-\int_\Delta (\chi\otimes\phi)1_{\{\tau< q\}}\,d\mu_\Delta$.
By~\eqref{eq:chi},
\begin{align*}
|\chi(y,\ell)-\chi(y',\ell)| 
& \le |\chi'(y)-\chi'(y')|+\sum_{k=0}^{\ell-1}|\phi(y,k)-\phi(y',k)| \\
& \ll \|v\|_\eta d_\Lambda(y,y')^\eta +  \|\phi\|_\eta \sum_{k=0}^{\ell-1} d_\Delta((y,k), (y',k))^\eta
 \\ & \ll \|v\|_\eta \tau(y) d_\Delta((y, \ell), (y', \ell))^\eta.
\end{align*}
Here we used that $\|\phi\|_\eta \ll \|v\|_\eta$ and $d_\Lambda(y,y') \leq d_\Delta((y,\ell), (y', \ell))$.

A simpler calculation shows that $|\chi(y,\ell)|\ll\tau(y)\|v\|_\eta$.
 It follows that 
  \(\|V\|_{\eta} \ll q \|v\|_\eta^2\).
  By Lemma~\ref{lem:phi},
  \[
    \big| \max_{k \le n} |S_kV| \big|_1
    \ll \|V\|_{\eta}\, n^{1/2}
    \ll q\|v\|_\eta^2\, n^{1/2}
    .
  \]
The result follows by taking \(q = n^{1/4}\).
 \end{proof}

\begin{rmk}  Nonuniformly expanding maps are mixing up to a finite cycle.  
When they are mixing (in particular, if $\gcd\{\tau(a):a\in\alpha\}=1$),
then we have formulas of Green-Kubo type for $\Sigma$ and $E$  in~\eqref{eq:drift}, namely
\[
\Sigma  = \int_\Lambda v\otimes v\,d\mu +\sum_{n=1}^{\infty}
\int_\Lambda \big(v\otimes (v\circ T^n)+(v\circ T^n)\otimes v\big)\,d\mu,
\quad
E  = \sum_{n=1}^{\infty} \int_\Lambda v\otimes (v\circ T^n)\,d\mu.
\]
\end{rmk}

\section{Families of nonuniformly expanding maps}
\label{sec:NUEn}

In this section, we prove the iterated WIP and iterated moment estimates 
for uniform families of nonuniformly expanding maps.

\subsection{Iterated WIP and iterated moments}
\label{sec:iterated}

Throughout, $T_n:\Lambda_n\to\Lambda_n$, $n\ge1$,
is a family of nonuniformly expanding maps as in Section~\ref{sec:NUE} with absolutely continuous ergodic $T_n$-invariant probability measures $\mu_n$.
To each $T_n$ there is associated an induced uniformly expanding map $F_n:Y_n\to Y_n$ with ergodic invariant probability measure $\mu_{Y_n}$ and a return time $\tau_n\in L^p(Y_n)$ where $p\ge2$.

We assume that $T_n$ is a {\em uniform family of order $p\ge2$}
 in the sense of~\cite{KKM18}.  This means 
that the expansion and distortion constants $C_1\ge1$, $\bbeta>1$, $\eta\in(0,1]$ for the induced maps $F_n$ can be chosen independent of $n$ and
that the family $\{\tau_n^p\}$ is uniformly integrable on $(Y_n,\mu_{Y_n})$,
i.e. \( \sup_n \int_{Y_n} \tau_n^p 1_{\{\tau_n \geq q\}} \, d\mu_{Y_n} \to 0 \)
as \(q \to \infty\).
Let $v_n:\Lambda_n\to\R^d$, $n\ge1$, be a family of observables with 
$\sup_{n\ge1}\|v_n\|_\eta<\infty$ and
$\int_{\Lambda_n}v_n\,d\mu_n=0$.

Let $f_{\Delta_n}:\Delta_n\to\Delta_n$ be the corresponding family of Young tower maps, with invariant probability measures $\mu_{\Delta,n}$ and semiconjugacies
$\pi_{\Delta_n}:\Delta_n\to\Lambda_n$.
In particular, $\mu_n=\pi_{\Delta_n*}\mu_{\Delta_n}$.

Define the lifted observables $\phi_n=v_n\circ\pi_{\Delta_n}:\Delta_n\to\R^d$.  By Section~\ref{sec:NUE}, we have the martingale-coboundary decompositions
\[
\phi_n=m_n+\chi_n\circ f_{\Delta_n}-\chi_n.
\]

\begin{prop} \label{prop:unif}
The family $\{|m_n|^2;\,n\ge1\}$ is uniformly integrable
on $(\Lambda_n,\mu_n)$.
\end{prop}

\begin{proof}
See~\cite[Proposition~4.3]{KKM18}.
\end{proof}

Abusing notation from Section~\ref{sec:NUE} slightly, we define
\[
S_kv_n=
\sum_{0\le j<k} v_n \circ T_n^j, \qquad
\BBS_kv_n=\sum_{0\le i<j<k} (v_n \circ T_n^i) \otimes (v_n \circ T_n^j).
\]
By uniformity, the constants $C$ in Section~\ref{sec:NUE} can be chosen independently of $n$.  Hence the next result is an immediate consequence of 
Theorem~\ref{thm:moments}:

\begin{cor}[Iterated moments] \label{cor:moments}
  For all \(n \geq 1\), 
  \[
    \big|\max_{k\le n}|S_kv_n|\big|_{L^{2(p-1)}(\mu_n)}\le Cn^{1/2}\|v_n\|_\eta
    , \qquad
    \big|\max_{k\le n}|\BBS_nv_n|\big|_{L^{p-1}(\mu_n)}\le Cn \|v_n\|_\eta^2
    . 
  \]

\vspace{-4ex} \qed
\end{cor}

Write
\begin{align} \label{eq:driftn}
  \Sigma_n & = 
  \lim_{k\to\infty}\frac1k
  \int_{\Delta_n}
S_kv_n
    \otimes
S_kv_n
  \,d\mu_{\Delta_n} 
  , \qquad
  E_n  =
  \lim_{k \to\infty} \frac1k
  \int_{\Delta_n} 
\BBS_kv_n
\,d\mu_{\Delta_n}
  .
\end{align}

\begin{cor} \label{cor:driftn}
The limits in~\eqref{eq:driftn}  exist for each $n$ and are given by
\[
\Sigma_n = \int_{\Delta_n} m_n \otimes m_n \, d\mu_{\Delta_n}, \qquad
E_n = \int_{\Delta_n} \chi_n \otimes \phi_n \, d\mu_{\Delta_n}.
\]
For $p>2$, the convergence is uniform in $n$.  
\end{cor}

\begin{proof}
This follows from Proposition~\ref{prop:drift}.
\end{proof}

Define $W_n\in D([0,1],\R^d)$, $\BBW_n\in D([0,1],\R^{d\times d})$ by
\[
  W_n(t) = \frac{1}{\sqrt n} \sum_{0\le j<nt}v_n \circ T_n^j, \qquad
  \BBW_n(t) = \frac1n \sum_{0\leq i<j<nt} (v_n \circ T_n^i) \otimes (v_n \circ T_n^j) .
\]
We can now state the main result of this section.
\begin{thm}[Iterated WIP] \label{thm:WIP}
  Suppose that $\lim_{n \to \infty} \Sigma_n = \Sigma$ and
  $\lim_{n \to \infty}E_n=E$. Then
  \[
(W_n,\BBW_n)\to_{\mu_n}(W,\BBW)\quad\text{as $n \to \infty\,$
  in $\,D([0,1], \R^d \times \R^{d \times d})$,}
\] where 
    $W$ is $d$-dimensional Brownian motion with covariance matrix $\Sigma$ and
    $\BBW(t)=\int_0^t W\otimes dW+Et$.
\end{thm}	

To prove Theorem~\ref{thm:WIP}, it is equivalent to show that 
$(Q_n,\BBQ_n)\to_{\mu_{\Delta_n}}(W,\BBW)$ where
\[
  Q_n(t) = \frac{1}{\sqrt n} \sum_{0\le j<nt}\phi_n \circ f_{\Delta_n}^j, \qquad
  \BBQ_n(t) = \frac1n \sum_{0\leq i<j<nt} (\phi_n \circ f_{\Delta_n}^i) \otimes (\phi_n \circ f_{\Delta_n}^j) .
\]
Define also
  $\BBM_n(t)
   = \frac1n \sum_{0\leq i<j<nt} (m_n \circ f_{\Delta_n}^i) \otimes (\phi_n \circ f_{\Delta_n}^j)
$
  .

\begin{lemma} \label{lem:KP}
  Suppose that $\lim_{n\to\infty}\Sigma_n=\Sigma$.
  Then $(Q_n,\BBM_n)\to_{\mu_{\Delta_n}}(W,\BBM)$
  in \(D([0,1], \R^d \times \R^{d \times d})\),
  where $\BBM(t)=\int_{0}^{t}W\otimes dW$. 
\end{lemma}

\begin{proof}  
We verify the hypotheses of Theorem~\ref{thm:mda}.
Hypothesis~(a) holds by Proposition~\ref{prop:unif}.
Next, by Proposition~\ref{prop:m}, writing $|\;|_2$ as shorthand for
$|\;|_{L^2(\mu_{\Delta_n})}$ and
$|\;|_{L^2(\mu_{Y_n})}$,
  \begin{align*}
\Big|\max_{k\le n}\big|\sum_{0\le j<k}(\phi_n-m_n)\circ f_{\Delta_n}^j\big|\Big|_2
& =
\Big|\max_{k\le n}|\chi_n\circ f_{\Delta_n}^k-\chi_n|\Big|_2
 \\ & 
\le C\|v_n\|_\eta(n^{1/4}+n^{1/2}|1_{\{\tau_n\ge n^{1/4}\}}\tau_n|_2).
\end{align*}
Since the family $\{\tau_n^2\}$ is uniformly integrable,
$n^{-1/2}\big|\max_{k\le n}|\sum_{0\le j<k}(\phi_n-m_n)\circ f_{\Delta_n}^j|\big|_2\to0$,
 verifying hypothesis~(b).

Finally, by Proposition~\ref{prop:breve}, for $t\in[0,1]$,
\[
\Big|\sum_{0\le j<nt}U_nL_n(m_n\otimes m_n)\circ f_{\Delta_n}^j-[nt]\Sigma_n\Big|_2\le
Cn^{1/2}\|v_n\|_\eta^2.
\]
Hypothesis~(c) follows.
\end{proof}

\begin{pfof}{Theorem~\ref{thm:WIP}}
  Write
\(
    \BBQ_n(t)-\BBM_n(t) = A_n(t) - B_n(t) 
    ,
  \)
  where 
  \begin{align*}
    A_n(t) & = \frac1n \sum_{0\le j<nt}(\chi_n\otimes\phi_n)\circ f_{\Delta_n}^j,
   \qquad B_n(t)  = \frac1n \chi_n \otimes \sum_{0\le j<nt}\phi_n\circ f_{\Delta_n}^j.
  \end{align*}
By Lemma~\ref{lem:KP}, it suffices to show that
\(
\sup_{t\in[0,1]}|A_n(t)-B_n(t)-tE_n|\to_{\mu_{\Delta_n}}0.
\)

Write $|\;|_q=|\;|_{L^q(\mu_{\Delta_n})}$.  
Since the  family
$\{\tau_n^2\}$ is uniformly integrable, it follows from
 Proposition \ref{prop:EeeE} that 
      \begin{align*}
        \big| \sup_{t \in [0,1]}|A_n(t)-t E_n|\big|_1
        & \ll \|v_n\|_\eta^2 \Big(
          n^{-1/4} + \int_{Y_n} \tau_n^2 1_{\{\tau_n \geq n^{1/4}\}} \, d\mu_{Y_n}
        \Big)
\to0
        .
      \end{align*}

Next,
      \(
        \sup_{t\in[0,1]} |B_n(t)|
        \leq |\chi_n| B'_n
      \)
where
      \(
        B_n' 
        = n^{-1} \max_{k \leq n}\big|\sum_{j=0}^{k-1}\phi_n\circ f_{\Delta_n}^j\big|
        .
      \)
      By Theorem~\ref{thm:moments},
      \({|B_n'|}_2 \ll n^{-1/2} \|v_n\|_\eta\ll n^{-1/2}\),
      so \(B_n' \to_{\mu_{\Delta_n}} 0\).
      Also, by Proposition~\ref{prop:m}, \({|\chi_n|}_1 \ll \|v_n\|_\eta=O(1)\). It follows that
      \(\sup_{t\in[0,1]} |B_n(t)| \to_{\mu_{\Delta_n}} 0\).
\end{pfof}

\begin{cor} \label{cor:WIP}
  Suppose that $\lim_{n \to \infty} \Sigma_n = \Sigma$ and
  $\lim_{n \to \infty}E_n=E$. 
Let $\lambda_n$ be a family of probability measures on $\Lambda_n$ absolutely continuous with respect to $\mu_n$. Suppose that the  densities 
$\rho_n=d\lambda_n/d\mu_n$ satisfy 
$\sup_n\int\rho_n^{1+\delta}\,d\mu_n<\infty$ for some $\delta>0$ and that
    $\BIG\inf_{N\ge1} \limsup_{n \to\infty}
      \int\Big|
      \frac1N \sum_{j=0}^{N-1} \rho_n \circ T_n^j 
      \,-\, 1
      \Big| \, d\mu_n = 0$.

Then
  $(W_n,\BBW_n)\to_{\lambda_n}(W,\BBW)$ as $n \to \infty$
  in \(D([0,1], \R^d \times \R^{d \times d})\)
 where 
    $W$ is $d$-dimensional Brownian motion with covariance matrix $\Sigma$ and
    $\BBW(t)=\int_0^t W\otimes dW+Et$.
\end{cor}

\begin{proof}
We verify the conditions of Remark~\ref{rmk:SDC}
with $\cB=D([0,1], \R^d \times \R^{d \times d})$ and
$d_\cB(u,v)=\sup_{t\in[0,1]}|u(t)-v(t)|$.
The result then follows from Theorem~\ref{thm:WIP}.

Conditions~(S1) and~(S5) of Remark~\ref{rmk:SDC} hold by assumption so it remains to verify (S4).
Define the sequence of random variables 
\[
R_n:\Lambda\to D([0,1],\R^d\times\R^{d\times d}), \qquad
R_n=(W_n,\BBW_n).
\]
We have $\sup_{t\in[0,1]}|(W_n\circ T_n)(t)-W_n(t)|\le 2n^{-1/2}|v_n|_\infty$.
Also,
\[
(\BBW_n\circ T_n)(t)-\BBW_n(t)
=n^{-1}\sum_{1\le i<nt}(v_n\circ T_n^i)\otimes (v_n\circ T_n^n)
-n^{-1}\sum_{1\le j<nt}v_n\otimes (v_n\circ T_n^j).
\]
Write $|\;|_1=|\;|_{L^1(\mu_n)}$.
By Corollary~\ref{cor:moments},
\begin{align*}
\big|\sup_{[0,1]}|\BBW_n\circ T_n-\BBW_n|\big|_1
& \le 4n^{-1}|v_n|_\infty \big|\max_{k\le n}|S_kv_n|\big|_1
 \ll n^{-1/2}\|v_n\|_\eta^2.
\end{align*}
Hence
\[
 |d_\cB( R_n \circ T_n,  R_n)|_1
\ll n^{-1/2}(|v_n|_\infty+\|v_n\|_\eta^2),
\]
verifying condition~(S4).
\end{proof}

\begin{rmk}
By Corollary~\ref{cor:moments},
$|N^{-1}\sum_{j=0}^{k-1}\rho_n\circ T_n^j-1|
\ll N^{-1/2}\|\rho_n\|_\eta$.
Hence a sufficient condition for the assumptions on $\rho_n$ in
Corollary~\ref{cor:WIP} is that
$\sup_n \|\rho_n\|_\eta<\infty$.  
\end{rmk}

\subsection{Existence of limits for \texorpdfstring{$\Sigma_n$ and $E_n$}{Sigma\_n and E\_n}}
\label{sec:En}

Theorem~\ref{thm:WIP} and Corollary~\ref{cor:WIP}
establish the iterated WIP subject to the existence of 
$\lim_{n\to\infty}\Sigma_n$ and
$\lim_{n\to\infty}E_n$.
In this subsection, we describe mild conditions under which these limits exist.

Let $(\Lambda,d_\Lambda)$ be a bounded metric space with finite Borel measure $\rho$.
We assume that $T_n$, $n\in\N\cup\{\infty\}$, is a uniform family as in Section~\ref{sec:iterated} but now of order $p>2$ and defined on the common space $\Lambda$.
In particular, each $T_n$ is a nonuniformly
expanding map as in Section~\ref{sec:NUE}, with absolutely continuous
ergodic $T_n$-invariant Borel probability measures $\mu_n$.
We suppose that $\mu_n$ is {\em statistically stable}:
$\mu_n\to_w\mu_\infty$ as $n\to\infty$.
Moreover, we require that 
\begin{align}
\label{eq:tech}
& \SMALL \int_\Lambda (v\circ T_\infty^j)(w\circ T_\infty^k)\,(d\mu_n-d\mu_\infty)\to 0
\qquad\text{and} \qquad
 T_n^j\to_{\mu_n}T_\infty^j \qquad\text{as $n\to\infty$}
\end{align}
for all $j,k\ge0$ and all $v,w:\Lambda\to\R$ H\"older.
(The second part of condition~\eqref{eq:tech} means that
$\mu_n\{y\in \Lambda:d_\Lambda(T_n^jy,T_\infty^jy)>a\}\to0$ for all $a>0$.)

Let $v_n\in C^\eta(\Lambda,\R^d)$, $n\in\N\cup\{\infty\}$, with $\int_\Lambda v_n\,d\mu_n=0$.
We assume that $\lim_{n\to\infty}\|v_n-v_\infty\|_\eta=0$.  

\begin{lemma} \label{lem:En}
Define
\(
S_nv_n = \sum_{0\le j<n}v_n\circ T_n^j
\) and
\(
\BBS_nv_n = \sum_{0\le i<j<n}
(v_n\circ T_n^i)
\otimes (v_n\circ T_n^j).
\)
Then the limits 
\begin{align*}  
\Sigma_n
 =\lim_{n\to\infty}\frac1n\int_\Lambda S_nv_n\otimes S_nv_n\,d\mu_n, \qquad
E_n = \lim_{n\to\infty}\frac1n \int_\Lambda \BBS_nv_n\,d\mu_n.
\end{align*}
exist for all $n\in\N\cup\{\infty\}$, and
$\lim_{n\to\infty}\Sigma_n=\Sigma_\infty$,
$\lim_{n\to\infty}E_n=E_\infty$.
\end{lemma}

\begin{proof}
The limits $\Sigma_n$ and $E_n$ exist for $n$ fixed by Corollary~\ref{cor:driftn}.
We refer to~\cite[Proposition~7.6]{KKM18} for the proof that 
$\lim_{n\to\infty}\Sigma_n=\Sigma_\infty$.
Here we show that $\lim_{n\to\infty}E_n=E_\infty$.

  Write $J_{n,n}=\int_\Lambda \BBS_nv_n\, d\mu_n$.
Let $\delta>0$.
  By Corollary~\ref{cor:driftn}, there exists $N\ge1$ such that
      $|N^{-1}J_{n,N}-E_n |<\delta$ for all $n\ge1$.
Hence,
\begin{equation} \label{eq:maybe}
|E_n-E_\infty|<2\delta+N^{-1}|J_{n,N}-J_{0,N}|.
\end{equation}
Next
  \[ 
    J_{n,N}-J_{0,N}
     = \int_\Lambda ( \BBS_Nv_n - \BBS_Nv_\infty) \, d \mu_n
    + \int_\Lambda \BBS_Nv_\infty \, (d\mu_n - d\mu_\infty).
  \]
  By condition~\eqref{eq:tech},
  \(\lim_{n\to\infty} \int_\Lambda \BBS_Nv_\infty \, (d\mu_n - d\mu_\infty) = 0\).
Also,
\[
|\BBS_Nv_n-\BBS_Nv_\infty|  \le
\sum_{0\le i<j<N}|(v_n\circ T_n^i)\otimes(v_n\circ T_n^j)-
(v_\infty\circ T_\infty^i)\otimes(v_\infty\circ T_\infty^j)| \le A_1+A_2
\]
where
\begin{align*}
A_1 & =
\sum_{0\le i<j<N}|(v_n\circ T_n^i)\otimes(v_n\circ T_n^j)-
(v_\infty\circ T_n^i)\otimes(v_\infty\circ T_n^j)|, \\
A_2 & = \sum_{0\le i<j<N}|(v_\infty\circ T_n^i)\otimes(v_\infty\circ T_n^j)-
(v_\infty\circ T_\infty^i)\otimes(v_\infty\circ T_\infty^j)|.
\end{align*}
Now,
\begin{align*}
A_1 & \le \sum_{0\le i<j<N} (|v_n|\circ T_n^i
|v_n-v_\infty|\circ T_n^j
+|v_n-v_\infty|\circ T_n^i |v_\infty|\circ T_n^j)
\\ & \le N^2(|v_n|_\infty+|v_\infty|_\infty)|v_n-v_\infty|_\infty.
\end{align*}
Also,
\[
A_2  \le N|v_\infty|_\infty\sum_{0\le j<N} 
|v_\infty\circ T_n^j-v_\infty\circ T_\infty^j|
\le N|v_\infty|_\infty|v_\infty|_\eta g_{n,N}
\]
where $g_{n,N}(y)=\sum_{j=0}^{N-1}d_\Lambda(T_n^jy,T_\infty^jy)^\eta$.
By~the assumption on $v_n$ and condition~\eqref{eq:tech}, we obtain that
$\lim_{n\to\infty}| \BBS_Nv_n - \BBS_Nv_\infty|_{L^1(\mu_n)}=0$.
Hence $\lim_{n\to\infty}J_{n,N}=J_{0,N}$ and so
$\limsup_{n\to\infty}|E_n-E_\infty|\le 2\delta$ by~\eqref{eq:maybe}.
  Since $\delta$  is arbitrary, the result follows.
\end{proof}

\subsection{Auxiliary properties}

Our results so far in this section on the iterated WIP and control of iterated moments verify the main hypotheses required to apply rough path theory in Part 2.
However, there remain two relatively minor hypotheses, Assumption~2.11 and Assumption~2.12(ii)(a) in~\cite{CFKMZsub} which we address now.
We continue to assume the set up of Subsection~\ref{sec:En} though we require weaker regularity assumptions on $v_n$: 
it suffices that $v_n\in L^\infty(\Lambda,\R^d)$, $n\ge1$, 
and $v_\infty\in C^\eta(\Lambda,\R^d)$ 
with $\lim_{n\to\infty}|v_n-v_\infty|_\infty=0$.
Fix $t\in[0,1]$, and
define $V_n=n^{-1}\sum_{j=0}^{[nt]-1}v_n\circ T_n^j$.

\begin{prop} \label{prop:aux}
\begin{itemize}
\item[(a)]
$\lim_{n\to\infty}\big|V_n -t\int_\Lambda v_\infty\,d\mu_\infty\big|_{L^p(\mu_n)}=0$.
\item[(b)] $\lim_{n\to\infty}\int_\Lambda v_n\otimes v_n\,d\mu_n
= \int_\Lambda v_\infty\otimes v_\infty\,d\mu_\infty$.
\end{itemize}
\end{prop}

\begin{proof}
(a)
Define $U_n=n^{-1}\sum_{j=0}^{[nt]-1}v_\infty \circ T_n^j$.
Then $|V_n-U_n|_\infty\le t|v_n-v_\infty|_\infty\to0$.
Since $v_\infty$ is H\"older, it follows from Corollary~\ref{cor:moments}
that  $|U_n-t\int_\Lambda v_\infty \,d\mu_n|_{L^p(\mu_n)}\ll \|v_\infty\|_\eta\, n^{-1/2}$.
By~\eqref{eq:tech}, $\int_\Lambda v_\infty \,d\mu_n\to \int_\Lambda v_\infty \,d\mu_\infty$ and
the result follows.
\\
(b) We have $\int_\Lambda|v_n\otimes v_n-v_\infty\otimes v_\infty|\,d\mu_n
\le (|v_n|_\infty+|v_\infty|_\infty)|v_n-v_\infty|_\infty\to0$.
Also, $\int_\Lambda v_\infty\otimes v_\infty(d\mu_n-d\mu_\infty)\to0$ by~\eqref{eq:tech}.
\end{proof}

\section{Examples}
\label{sec:ex}

In this section, we consider examples of nonuniformly expanding dynamics, including families of intermittent maps~\eqref{eq:LSV} discussed in the introduction, covered by the results in this paper.

\subsection{Application to intermittent maps}
\label{sec:PM}

Fix a family of intermittent maps $T_n:[0,1]\to[0,1]$, $n\in\N\cup\{\infty\}$, as in~\eqref{eq:LSV}
with parameters $\gamma_n\in(0,\frac12)$ such that 
$\lim_{n\to\infty}\gamma_n=\gamma_\infty$.  
By~\cite[Example~5.1]{KKM17}, 
$T_n$ is a uniform family of nonuniformly expanding maps of order $p$ for
all $p\in(2,\gamma_\infty^{-1})$.  
By~\cite{BaladiTodd16,Korepanov16}, $\mu_n$ is strongly statistically stable.
That is, the densities $h_n=d\mu_n/d\Leb$ satisfy
$\lim_{n\to\infty}\int_\Lambda|h_n-h_\infty|\,d\Leb=0$.
Using this property, conditions~\eqref{eq:tech} are easily verified.

Hence our main results on control of iterated moments (Corollary~\ref{cor:moments}) and the iterated WIP (Theorem~\ref{thm:WIP} and Lemma~\ref{lem:En}) hold  for families of intermittent maps $T_n$ and H\"older observables $v_n:[0,1]\to\R^d$ with $\int v_n\,d\mu_n=0$ and $\lim_{n\to\infty}\|v_n-v_\infty\|_\eta=0$.
Also the auxiliary properties in Proposition~\ref{prop:aux} are satisfied.
This leads via Part 2 to homogenization results $\hat x_n\to_{\mu_n}X$ for fast-slow systems~\eqref{eq:fs}.  Since $\mu_n(\hat x_n\in B)-\mu_\infty(\hat x_n\in B)=\int_\Lambda 1_{\{\hat x_n\in B\}}(h_n-h_\infty)\,d\Leb$ for suitable subsets $B\subset D([0,1],\R^d)$, it follows from strong statistical stability that $\hat x_n\to_{\mu_\infty}X$.

In the remainder of this subsection, we show that all our results remain valid when $\mu_n$ is replaced by Lebesgue measure.  (We continue to assume that the observables  $v_n$ are centered with respect to $\mu_n$, so $\int v_n\,d\mu_n=0$.)

The densities $h_n=d\mu_n/d\Leb$ are uniformly bounded below
(see~\cite[Lemma~2.4]{LSV99} for explicit lower bounds).
Hence it is immediate that the moment estimates in Corollary~\ref{cor:moments} hold with $\mu_n$ changed to $\Leb$.
Since $\Leb$ is not invariant, the following nonstationary version of the moment estimates is required in Part~2:
\begin{prop} \label{prop:momentsLeb}
\(
\big|\sum_{\ell\le j<k}v_n\circ T_n^j\big|_{L^{2(p-1)}(\Leb)}\le C(k-\ell)^{1/2}\|v_n\|_\eta\) and
\newline \(\big|\sum_{\ell\le i<j<k}(v_n\circ T_n^i)\otimes(v_n\circ T_n^j)\big|_{L^{p-1}(\Leb)}\le C(k-\ell)\|v_n\|_\eta^2
\)
for all $0\le \ell<k<n$.
\end{prop}

\begin{proof}
Since $\mu_n$ is $T_n$-invariant, it follows from
Corollary~\ref{cor:moments} that
\begin{align*} 
\Big|\sum_{\ell\le j<k}v_n\circ T_n^j\Big|_{L^{2(p-1)}(\mu_n)}
 =\Big|\sum_{0\le j<k-\ell} & v_n\circ T_n^j\Big|_{L^{2(p-1)}(\mu_n)}
\ll (k-\ell)^{1/2}\|v_n\|_\eta
\\ 
\Big|\sum_{\ell\le i<j<k}(v_n\circ T_n^i)\otimes(v_n\circ T_n^j)\Big|_{L^{p-1}(\mu_n)}
& = \Big|\sum_{0\le i<j<k-\ell}(v_n\circ T_n^i)\otimes(v_n\circ T_n^j)\Big|_{L^{p-1}(\mu_n)}
 \\ & 
\ll (k-\ell)\|v_n\|_\eta^2.
\end{align*}
Now use that the densities $h_n$ are uniformly bounded below.
\end{proof}

Next we turn to the iterated WIP.  Defining $\Sigma_n$ and $E_n$ as in~\eqref{eq:driftn} for $n\in\N\cup\{\infty\}$, we already have that
$\lim_{n\to\infty}\Sigma_n=\Sigma_\infty$ and
$\lim_{n\to\infty}E_n= E_\infty$ by Lemma~\ref{lem:En}.

\begin{prop} \label{prop:WIPLeb}
$(W_n,\BBW_n)\to_\Leb(W,\BBW)$ as $n \to \infty$
  in $D([0,1], \R^d \times \R^{d \times d})$,
where
    $W$ is $d$-dimensional Brownian motion with covariance matrix $\Sigma_\infty$ and
    $\BBW(t)=\int_0^t W\otimes dW+E_\infty t$.
\end{prop}

\begin{proof}
By Theorem~\ref{thm:WIP}, 
$(W_n,\BBW_n)\to_{\mu_n}(W,\BBW)$ as $n \to \infty$
  in $D([0,1], \R^d \times \R^{d \times d})$.
To pass from $\mu_n$ to $\Leb$,
 we apply Corollary~\ref{cor:WIP}.
Let $\rho_n=h_n^{-1}=d\Leb/d\mu_n$.  
Then $\sup_n|\rho_n|_\infty<\infty$.
To deal with the remaining assumption in Corollary~\ref{cor:WIP},
write
\[
\int_\Lambda\Big|\frac1N\sum_{j=0}^{N-1}\rho_n\circ T_n^j-1\Big|\,d\mu_n
\le I_1(N,n)+I_2(N,n)+I_3(N,n)+I_4(N)
\]
where
\begin{align*}
I_1 & 
= \int_\Lambda\frac1N\Big|\sum_{j=0}^{N-1}(\rho_n\circ T_n^j-\rho_\infty\circ T_n^j)\Big|\,d\mu_n, \quad
I_2  =
\int_\Lambda\frac1N\sum_{j=0}^{N-1}\rho_\infty\circ T_n^j|h_n-h_\infty|\,d\Leb, \\
I_3  &=
\int_\Lambda\frac1N\Big|\sum_{j=0}^{N-1}(\rho_\infty\circ T_n^j-\rho_\infty\circ T_\infty^j)\Big|\,d\mu_\infty, \quad
I_4  =
\int_\Lambda\Big|\frac1N\sum_{j=0}^{N-1}\rho_\infty\circ T_\infty^j-1\Big|\,d\mu_\infty.
\end{align*}

Fix $N\ge1$.
By $T_n$-invariance of $\mu_n$, 
\[
I_1(N,n)\le 
\int_\Lambda|\rho_n-\rho_\infty|\,d\mu_n=\int_\Lambda|h_n-h_\infty|\rho_\infty\,d\Leb
\le |\rho_\infty|_\infty\int_\Lambda|h_n-h_\infty|\,d\Leb,
\]
and also
\[
I_2(N,n)
\le |\rho_\infty|_\infty\int_\Lambda|h_n-h_\infty|\,d\Leb.
\]
By boundedness of $\rho_\infty$ and strong statistical stability,
$\lim_{n\to\infty}I_1(N,n)
=\lim_{n\to\infty}I_2(N,n)=0$.
By continuity of $\rho_\infty$ and the dominated convergence theorem,
$\lim_{n\to\infty}I_3(N,n)=0$.
Hence for each fixed $N\ge1$,
\[
\limsup_{n\to\infty}
\int_\Lambda\Big|\frac1N\sum_{j=0}^{N-1}\rho_n\circ T_n^j-1\Big|\,d\mu_n
\le I_4(N).
\]
By the mean ergodic theorem, $\lim_{N\to\infty}I_4(N)=0$.
Hence $\limsup_{n\to\infty}
\int_\Lambda\Big|\frac1N\sum_{j=0}^{N-1}\rho_n\circ T_n^j-1\Big|\,d\mu_n\to 0$ as $N\to\infty$.
This verifies the final assumption in Corollary~\ref{cor:WIP} completing the proof.
\end{proof}

Finally, we consider the analogue of Proposition~\ref{prop:aux} with $\mu_n$ replaced by $\Leb$ where appropriate.  
Again we can relax the assumptions on $v_n$;
it suffices that $v_n\in L^\infty(\Lambda,\R^d)$, $n\ge1$, 
and $v_\infty\in C^\eta(\Lambda,\R^d)$ 
with $\lim_{n\to\infty}|v_n-v_\infty|_\infty=0$.

Recall that $V_n=n^{-1}\sum_{j<[nt]}v_n\circ T_n^j$ where $t\in[0,1]$ is fixed.
\begin{prop} \label{prop:auxLeb}
\begin{itemize}
\item[(a)]
$\lim_{n\to\infty}\big|V_n -t\int_\Lambda v_\infty\,d\mu_\infty\big|_{L^p(\Leb)}=0$.
\item[(b)] $\lim_{n\to\infty}n^{-1}\sum_{j<n}\int_\Lambda (v_n\otimes v_n)\circ T_n^j\,d\Leb
= \int_\Lambda v_\infty\otimes v_\infty\,d\mu_\infty$.
\end{itemize}
\end{prop}

\begin{proof}
(a) Using again that the densities $h_n=d\mu_n/d\Leb$ are uniformly bounded below,
\[
\int_\Lambda\Big|V_n- t\int_\Lambda v_\infty\,d\mu_\infty\Big|\,d\Leb
\ll \int_\Lambda\Big|V_n- t\int_\Lambda v_\infty\,d\mu_\infty\Big|\,d\mu_n\to0
\]
by Proposition~\ref{prop:aux}(a).

\noindent(b) Set $w_n=v_n\otimes v_n-\int_\Lambda v_n\otimes v_n \,d\mu_n$.
Then $w_n\in C^\eta(\Lambda,\R^{d\times d})$ with $\int_\Lambda w_n\,d\mu_n=0$ and
\begin{align*}
\Big|n^{-1}\sum_{j<n}\int_\Lambda(v_n\otimes v_n)\circ T_n^j\,d\Leb & -\int_\Lambda v_n\otimes v_n\,d\mu_n\Big|
 \le 
n^{-1}\int_\Lambda\Big|\sum_{j<n}w_n\circ T_n^j\Big|\,d\Leb
\\ & \ll n^{-1}\int_\Lambda\Big|\sum_{j<n}w_n\circ T_n^j\Big|\,d\mu_n
 \ll n^{-1/2}\|w_n\|_\eta\to0
\end{align*}
by Corollary~\ref{cor:moments}.
\end{proof}

\subsection{Further examples}

In~\cite{KKM18}, the WIP and estimates of ordinary moments were obtained for many examples of nonuniformly expanding dynamics.  We now obtain the corresponding iterated results.

Revisiting~\cite[Example~5.2]{KKM17} and~\cite[Example~4.10, Example~7.3]{KKM18}, we consider families of quadratic maps $T_n:[-1,1]\to[-1,1]$, $n\in\N\cup\{\infty\}$, given by $T_n(x)=1-a_n x^2$, $a_n\in[0,2]$
with $\lim_{n\to\infty}a_n=a_\infty$.  Fixing $b,c>0$ we assume that the Collet-Eckmann condition $|(T_n^k)'(1)|\ge ce^{bk}$ holds for all $k\ge0$, $n\ge1$.\footnote{There is a typo in~\cite[Example~4.10]{KKM18} where $ce^{bn}$ should be $ce^{bk}$.}
The set of parameters such that this Collet-Eckmann condition holds
has positive Lebesgue measure for $b,c$ sufficiently small.
Moreover $T_n$ is a uniform family of nonuniformly expanding maps of order $p$ (for any $p$) and satisfies strong statistical stability.
Hence we obtain control of iterated moments (Corollary~\ref{cor:moments}) and the iterated WIP (Theorem~\ref{thm:WIP} and Lemma~\ref{lem:En}) for H\"older observables $v_n:[-1,1]\to\R^d$ with $\int_\Lambda v_n\,d\mu_n=0$ and $\lim_{n\to\infty}|v_n-v_\infty|_\infty=0$.

Revisiting~\cite[Example~5.4]{KKM17} and~\cite[Example~4.11, Example~7.3]{KKM18}, we consider families of Viana maps $T_n:S^1\times\R\to S^1\times\R$, $n\in\N\cup\{\infty\}$. 
Again, we obtain control of iterated moments and the iterated WIP.

In both sets of examples, we obtain homogenization results $\hat x_n\to_{\mu_n}X$ by Part~2
and $\hat x_n\to_{\mu_\infty}X$ by strong statistical stability as explained at the beginning of Subsection~\ref{sec:PM}.

\section{Families of nonuniformly expanding semiflows}
\label{sec:flow}

In this section, we consider uniform families of nonuniformly expanding semiflows.
These are modelled as suspensions over uniform families of nonuniformly expanding maps.
In keeping with the program of~\cite{MS11}, no mixing assumptions are imposed on the semiflows.

Specifically, let $T_n:\Lambda_n\to\Lambda_n$, $n\ge1$, be a uniform family of nonuniformly expanding maps of order $p\ge2$ as in Section~\ref{sec:NUEn} with ergodic invariant probability measures $\mu_{\Lambda_n}$.
Let $h_n:\Lambda_n\to\R^+$ be a family of roof functions satisfying 
\mbox{$\sup_n\|h_n\|_\eta<\infty$} and $\inf_n \inf h_n>0$.  
For each $n\ge1$, define the suspension
\[
  \Omega_n=\Lambda_n^{h_n} = \{ (x,u) \in \Lambda_n \times \R :
  0 \leq u \le h_n(x)\}/\sim\, ,
\quad  (x, h_n(x)) \sim (T_n x,0).
\]
The suspension flow \(g_{n,t}:\Omega_n\to\Omega_n\) is given by \(g_{n,t}(x,u) = (x,u+t)\)
computed modulo identifications.
Let \(\bar h_n = \int_{\Lambda_n} h_n \, d\mu_{\Lambda_n}\).
Then
\(
  \mu_n =\mu_{\Lambda_n}^{h_n}
  = (\mu_{\Lambda_n} \times \text{Lebesgue}) 
  / \bar h_n
\)
is an ergodic \(g_{n,t}\)-invariant probability measure 
on $\Omega_n$. 
We call $g_{n,t}:\Omega_n\to\Omega_n$ a {\em uniform family of nonuniformly expanding semiflows of order $p$}.

To simplify the statement of results in this section, we denote by \(C\) various constants
depending continuously on the data associated with $T_n:\Lambda_n\to\Lambda_n$ as well as $\sup_n\|h_n\|_\eta$.

For $v:\Omega_n\to\R^d$, define 
\[
  \|v\|_\eta = |v|_\infty + |v|_\eta
  \quad \text{where} \quad 
  |v|_\eta = 
\sup_{(x,u)\neq(x',u)\in \Omega_n}\frac{|v(x,u)-v(x',u)|}{d_{\Lambda_n}(x,x')^\eta}.
\]

\subsection{Moment estimates}
\label{sec:momentflow}

As in Sections~\ref{sec:NUE} and~\ref{sec:NUEn}, for uniform moment estimates it suffices to consider a fixed uniformly expanding semiflow $g_t:\Omega\to\Omega$.
The main result in this section, Theorem~\ref{thm:momentflow}, establishes the desired moment estimates.  We also collect together some other results that fit best into the fixed semiflow setting.

Given $v:\Omega\to\R^d$, define 
the {\em induced observable}
\[
\tilde v:\Lambda\to\R, \qquad
\SMALL \tilde v(x)=\int_0^{h(x)}v(x,u)\,du.
\]

\begin{prop} \label{prop:tildev}
$|\tilde v|_\infty\le |h|_\infty|v|_\infty$ and
$\|\tilde v\|_\eta\le \|h\|_\eta \|v\|_\eta$.
\end{prop}

\begin{proof}
The estimate for $|\tilde v|_\infty$ is immediate.
Also, for $x,x'\in \Lambda$ with $h(x)\le h(x')$,
\begin{align*}
|\tilde v(x)-\tilde v(x')| & \le |h(x)-h(x')||v|_\infty
+\int_0^{h(x)}|v(x,u)-v(x',u)|\,du
\\ & \le |h|_\eta|v|_\infty d_\Lambda(x,x')
+\int_0^{h(x)}|v|_\eta\, d_\Lambda(x,x')\,du \le \|h\|_\eta\|v\|_\eta d_\Lambda(x,x'),
\end{align*}
completing the proof.
\end{proof}

Define
\[
S_t=\int_0^t v\circ g_s\,ds, \qquad
\BBS_t=\int_0^t\int_0^s (v\circ g_r) \otimes (v\circ g_s) \,dr\,ds
\]
on $\Omega$.
Also, for the induced observable $\tilde v:\Lambda\to\R^d$, define
\[
\tS_n(x,u)=\sum_{0\le j<n}\tilde v(T^jx), \qquad
\tBBS_n(x,u)=\sum_{0\le i<j<n}\tilde v(T^ix)\otimes \tilde v(T^jx),
\quad (x,u)\in\Omega.
\]

We introduce
the {\em lap number} $N(t):\Omega\to\N$, $t\ge0$,
\[
\SMALL N(t)(x,u)=\max\big\{n\ge0:\sum_{j=0}^{n-1}h(T^jx)\le u+t\big\}.
\]
Also, define 
\[
\SMALL H:\Omega\to\R^d, \qquad H(x,u)=\int_0^u v(x,s)\,ds.
\]

\begin{prop} \label{prop:flow}
For all $t\ge0$,
\begin{align*}
& |S_t-\tS_{N(t)}|_\infty\le 2|h|_\infty|v|_\infty,  \\
& \SMALL
   |\BBS_t-\tBBS_{N(t)}-\int_0^t(H\otimes v)\circ g_s\,ds|\le
2|h|_\infty|v|_\infty|\tS_{N(t)}|+2|h|_\infty^2|v|_\infty^2.
\end{align*}
\end{prop}

\begin{proof}
We use formal calculations from the proof of~\cite[Proposition~7.5]{KM16}, focusing on the precise estimates.

First, $S_t=\tS_{N(t)}+ H\circ g_t-H$.  Hence 
\(
|S_t-\tS_{N(t)}|_\infty \le 2|H|_\infty\le 2|h|_\infty|v|_\infty.
\)

Second, writing $T_n = \inf \{ t' \geq 0: N(t') = n \}$, we observe that
\[
    \tBBS_{N(t)}
    = \int_0^{T_{N(t)}} \tS_{N(s)} \otimes (v \circ g_s) \, ds
    = \int_0^{t} \tS_{N(s)} \otimes (v \circ g_s) \, ds
    - \tS_{N(t)} \otimes (H \circ g_t)
    .
\]
Hence
\begin{align*}
\BBS_t
& \SMALL =\int_0^t\tS_{N(s)}\otimes (v\circ g_s) \,ds
+\int_0^t (H\otimes v)\circ g_s \,ds
-H\otimes\int_0^t v\circ g_s\,ds
 \\[.75ex] & \SMALL 
= \tBBS_{N(t)}+
\int_0^t (H\otimes v)\circ g_s \,ds+ K(t),
\end{align*}
where
$K(t)=
\tS_{N(t)}\otimes (H\circ g_t) - H\otimes S_t$.
We have
\[
|K(t)|  \le
|\tS_{N(t)}||H|_\infty
+|H|_\infty|S_t|
\le
|\tS_{N(t)}||h|_\infty|v|_\infty
+|h|_\infty|v|_\infty(|\tS_{N(t)}|+2|h|_\infty|v|_\infty).
\]
The result follows.
\end{proof}

\begin{thm}[Iterated moments] \label{thm:momentflow}  
For all $\TT\ge0$,
\[
\Big|\sup_{t\in[0,\TT]}|S_t|\Big|_{L^{2(p-1)}(\Omega)}\le C\TT^{1/2}\|v\|_\eta, \qquad
\Big|\sup_{t\in[0,\TT]}|\BBS_t|\Big|_{L^{p-1}(\Omega)}\le C\TT \|v\|_\eta^2.
\]
\end{thm}

\begin{proof}
The estimates are trivial for $t\in[0,1]$
(since $t\le t^{1/2}$) so we restrict to $\TT\ge1$, $t\in[1,\TT]$.

Since $\inf h>0$, it is immediate (\cite[Proposition~7.4]{KM16}) that
\begin{align} \label{eq:lap}
|N(t)|_\infty\le C_0\,  t \quad\text{for all $t\ge1$,}
\end{align}
where $C_0=(\inf h)^{-1}+1$.

By~\eqref{eq:lap},
\begin{align*}
\int_\Omega\sup_{1\le t\le \TT}|\tS_{N(t)}|^{2(p-1)}\,d\mu
& \le \int_\Omega\max_{k\le C_0\TT}|\tS_k|^{2(p-1)}\,d\mu
\\ & \le \bar h^{-1}|h|_\infty 
\int_\Lambda\max_{k\le C_0\TT}\Big|\sum_{0\le j<k}\tilde v\circ T^j\Big|^{2(p-1)}\,d\mu_\Lambda.
\end{align*}
Hence by Theorem~\ref{thm:moments},
\[
\Big|\sup_{1\le t\le \TT}|\tS_{N(t)}|\Big|_{L^{2(p-1)}(\Omega)}
\ll
\Big|\max_{k\le C_0\TT}\Big|\sum_{0\le j<k}\tilde v\circ T^j\Big|\Big|_{L^{2(p-1)}(\Lambda)}
\ll \TT^{1/2}\|\tilde v\|_\eta\ll \TT^{1/2}\|v\|_\eta.
\]
Similarly, \(\big|\sup_{1\le t\le \TT}|\tBBS_{N(t)}|\big|_{L^{p-1}(\Omega)}\ll \TT\|v\|_\eta^2\).
Also, $\int_0^\TT|(H\otimes v)\circ g_s|\,ds\le \TT|h|_\infty|v|_\infty^2$ so the result follows from Proposition~\ref{prop:flow}.
\end{proof}

\begin{cor} \label{cor:moment}
For all $\TT\ge0$,
\[
\Big|\sup_{t\in[0,\TT]}\Big|\int_0^t(H\otimes v)\circ g_s\,ds-t\int_\Omega H\otimes v\,d\mu\Big|\Big|_{L^{2(p-1)}(\Omega)}\le C\TT^{1/2}\|v\|_\eta^2.
\]
\end{cor}

\begin{proof}
Using the $S_t$ estimate in Theorem~\ref{thm:momentflow} with
$v$ replaced by $H\otimes v - \int_\Omega H\otimes v\,d\mu$,
we obtain $\big|\sup_{t\in[0,\TT]}|\int_0^t(H\otimes v)\circ g_s\,ds-t\int_\Omega H\otimes v\,d\mu|\big|_{2(p-1)}\ll \TT^{1/2}\|H\otimes v\|_\eta$.
In addition,
$\|H\otimes v\|_\eta\le \|H\|_\eta\|v\|_\eta
\le |h|_\infty\|v\|_\eta^2$.
\end{proof}

\begin{lemma} \label{lem:Cs}
For all $s\in[0,\inf h]$, $n\ge1$, 
\[
\Big|\sup_{t\in[0,1]}|\tS_{[nt]}\circ g_s-\tS_{[nt]}|\Big|_\infty \le 2|h|_\infty |v|_\infty,
\quad 
\Big|\sup_{t\in[0,1]}|\tBBS_{[nt]}\circ g_s-\tBBS_{[nt]}|\Big|_{L^{2(p-1)}(\Omega)} \le C n^{1/2}\|v\|_\eta^2.
\]
\end{lemma}

\begin{proof}
The random variable $N(s)(x,u)$ lies in $\{0,1\}$ due to the restriction on $s$.
If $N(s)(x,u)=0$, then $g_s(x,u)=(x,u+s)$.
Now, $\tS_n$ and $\tBBS_n$ are independent of $u$, and so
$\tS_{[nt]}\circ g_s\equiv\tS_{[nt]}$ and 
$\tBBS_{[nt]}\circ g_s\equiv\tBBS_{[nt]}$ for all $n,t$ and all $s,x,u$ with
$N(s)(x,u)=0$.

Hence we may suppose for the remainder of the proof that $N(s)\equiv1$
in which case $g_s(x,u)=(fx,u+s-h(x))$.
Then,
\[
|\tS_{[nt]}\circ g_s(x,u)-\tS_{[nt]}(x,u)|  =
\Big|
\sum_{0\le j<[nt]} (\tilde v(T^{j+1}x)
- \tilde v(T^jx))\Big|
\le 2|\tilde v|_\infty \le 2|h|_\infty|v|_\infty.
\]

Next, 
\begin{align*}
|\tBBS_{[nt]}\circ g_s(x,u)-\tBBS_{[nt]}(x,u)| & =
\Big|
\sum_{0\le i< j<[nt]} \big(\tilde v(T^{i+1}x) \otimes \tilde v(T^{j+1}x)
- \tilde v(T^ix) \otimes \tilde v(T^jx)\big)\Big|
\\ & = \Big| \sum_{1\le i< [nt]} 
 \tilde v(T^ix) \otimes \tilde v(T^{[nt]}x)
-\sum_{1\le j< [nt]} 
\tilde v \otimes \tilde v(T^jx) \Big|
\\ & \le 2|\tilde v|_\infty \Big|\Big(\sum_{0\le j< [nt]-1}\tilde v\circ T^j\Big)(fx)\Big|
\end{align*}
Hence
$\Big|\sup_{t\in[0,1]}|\tBBS_{[nt]}\circ g_s-\tBBS_{[nt]}|\Big|_{2(p-1)} \ll  
n^{1/2}\|\tilde v\|_\eta^2  \ll
n^{1/2}\|v\|_\eta^2$ by 
Theorem~\ref{thm:moments}.
\end{proof}

\begin{prop} 
  \label{prop:laps}
      \(
        \big| 
          \sup_{t \le \TT} |N(t) - t / \bar h| 
        \big|_{L^{2(p-1)}(\Omega)}
        \leq C \TT^{1/2}
      \) for \(\TT \geq 1\).
\end{prop}

\begin{proof}
Let $S_kh=\sum_{j=0}^{k-1}h\circ T^j$.
By definition of $N(t)$,
\[
S_{N(t)(x,u)}h(x)\le u+t < S_{N(t)(x,u)+1}h(x).
\]
Hence 
\(
-S_{N(t)(x,u)}h(x)-|h|_\infty< -t \le  -S_{N(t)(x,u)}h(x)+|h|_\infty,
\)
so
\[
|N(t)(x,u)-t/\bar h|\le \{|S_{N(t)(x,u)}h(x)-N(t)(x,u)\bar h|+2|h|_\infty\}/\bar h.
\]
  By~\eqref{eq:lap}, for all \((x,u) \in \Omega\),
  \[
    \sup_{t \le \TT} |N(t)(x,u) - t / \bar h|
    \ll \max_{k \leq C_0\TT}  \big| \SMALL S_kh(x) - k \bar h \big| + 1
    .
  \]
Hence by Theorem~\ref{thm:moments},
  \[
      \big|
    \sup_{t \le \TT} |N(t) - t / \bar h|
\big|_{L^{2(p-1)}(\Omega)}
    \ll 
      \big|\max_{k \leq C_0\TT} | S_kh - k \bar h |\big|_{L^{2(p-1)}(\Lambda)}
    \ll \TT^{1/2}
    ,
  \]
as required.
\end{proof}

\subsection{Iterated weak invariance principle}
\label{sec:WIPflow}

Let $g_{n,t}:\Omega_n\to\Omega_n$ be a uniform family of nonuniformly expanding semiflows of order $p\ge 2$.
Let $v_n:\Omega_n\to\R^d$, $n\ge1$, be a family of observables with 
$\sup_{n\ge1}\|v_n\|_\eta<\infty$ and
$\int_{\Omega_n}v_n\,d\mu_n=0$.
The corresponding family of induced observables $\tilde v_n:\Lambda_n\to\R^d$ satisfies
$\sup_{n\ge1}\|\tilde v_n\|_\eta<\infty$ and
$\int_{\Lambda_n}v_n\,d\mu_{\Lambda_n}=0$.
Define $\Sigma_n$ and $E_n$ in terms of $\tilde v_n$ as in~\eqref{eq:driftn}.
Also define $H_n(x,u)=\int_0^u v_n(x,s)\,ds$.

In this section, we prove an iterated WIP for 
the processes $W_n \in C([0,1], \R^d)$ and
$\BBW_n \in C([0,1],\R^{d \times d})$ on $\Omega_n$ given by
\[
  W_n(t) = \frac{1}{\sqrt n} \int_0^{ n t } v_n \circ g_{n,s} \, ds
  , \qquad
  \BBW_n(t) = \int_0^t W_n(s) \otimes dW_n(s)
  .
\]

First, we consider the processes $\tW_n\in D([0,1],\R^d)$, $\tBBW_n\in D([0,1],\R^{d\times d})$
\[
  \tW_n(t)(x,u) = \frac{1}{\sqrt n} \sum_{j=0}^{[n t]-1}\tilde v_n(T_n^jx), \qquad
  \tBBW_n(t)(x,u) = \frac1n \sum_{0\leq i<j<[nt]} \tilde v_n(T_n^ix) \otimes \tilde v_n(T_n^jx) ,
\] 
defined on \(\Omega_n\).
Let $N_n(t)$ denote the lap numbers  corresponding to the semiflows $g_{n,t}$  on $\Omega_n$.  Also define
\[
\gamma_n\in D([0,1],\R), \qquad
\gamma_n(t)=n^{-1}N_n(nt).
\]

\begin{prop}
  \label{prop:gg}
Suppose that $E_n \to E$,
  $\Sigma_n\to\Sigma$, $\bar h_n\to\bar h$. Then
  \[
(\tW_n, \tBBW_n)\circ \gamma_n \to_{\mu_n} (\bar h^{-1/2}\tW,\bar h^{-1}\tBBW) 
\quad\text{in}\;\, D([0,1], \R^d \times \R^{d \times d}),
\]
  where
    \(\tW\) is a \(d\)-dimensional Brownian motion with covariance matrix
      \(\Sigma\) and
    \(\tBBW(t) = \int_0^t \tW \otimes d\tW + E t\).
\end{prop}

\begin{proof}
Choose $c_0>0$ such that $h_n\ge c_0$ for all $n$.
  Let \(\lambda_n\) be the sequence of probability measures on \(\Omega_n\) 
supported on $\Lambda_n\times[0,c_0]$ 
with density
  \(
    \rho_n = d\lambda_n/d\mu_n = 1_{\Lambda_n\times[0,c_0]}/c_0.
  \)

  The process
  \((\tW_n, \tBBW_n)\) on $(\Omega_n,\lambda_n)$ has the same distribution as 
the process \((\tW_n, \tBBW_n)|_{\Lambda_n}\) on
  \((\Lambda_n, \mu_{\Lambda_n})\), so by Theorem~\ref{thm:WIP},
  \((\tW_n, \tBBW_n) \to_{\lambda_n} (\tW, \tBBW)\).
  
  By Lemma~\ref{lem:Cs}, for each $s\in[0,c_0]$,
\[
\big|\sup_{t\in[0,1]}|\tW_{[nt]}\circ g_{n,s}-\tW_{[nt]}|\big|_\infty \ll n^{-1/2},
\qquad 
\big|\sup_{t\in[0,1]}|\tBBW_{[nt]}\circ g_{n,s}-\tBBW_{[nt]}|\big|_{L^1(\mu_n)} \ll n^{-1/2}.
\]
Also, $|\rho_n|_\infty\le 1/c_0$ and $|\rho_n|_\eta=0$ so by Theorem~\ref{thm:momentflow} with $v=\rho_n-1$,
\begin{align*} 
\SMALL |\int_0^\TT \rho_n \circ g_{n,t} \, dt - \TT|_{L^2(\mu_n)} \ll \TT^{1/2}.
\end{align*}
  We have verified the assumptions of Lemma~\ref{lem:SDC},
  and it follows that \((\tW_n, \tBBW_n) \to_{\mu_n} (\tW, \tBBW)\).

Let $\gamma(t)=t\bar h^{-1}$.
By Proposition~\ref{prop:laps},
\[
\big|\sup_{t\le1}|\gamma_n(t)-\gamma(t)|\big|_{L^1(\mu_n)}=n^{-1}\big|\sup_{t\le 1}|N_n(nt)-nt\bar h^{-1}|\big|_{L^1(\mu_n)} = O(n^{-1/2}).
\]
Since $\gamma$ is not random it follows that
$(\tW_n,\tBBW_n,\gamma_n)\to_{\mu_n}(\tW,\tBBW,\gamma)$.
By the continuous mapping theorem, 
\begin{align*}
(\tW_n,\tBBW_n)\circ \gamma_n \to_{\mu_n}
(\tW,\tBBW)\circ\gamma=(\bar h^{-1/2}\tW,\bar h^{-1}\tBBW),
\end{align*}
as required.
\end{proof}

\begin{thm}[Iterated WIP] \label{thm:WIPflow}
  Suppose that \(\lim_{n\to\infty}\Sigma_n = \Sigma\), \(\lim_{n\to\infty}E_n = E\), \(\lim_{n\to\infty}\bar h_n = \bar h\)
  and \( \lim_{n\to\infty}\int_{\Omega_n} H_n\otimes v_n\,d\mu_n = E'\).
  Then \[(W_n, \BBW_n) \to_{\mu_n} (W, \BBW)
  \quad\text{in}\;\, D([0,1], \R^d \times \R^{d \times d}),\]
  where
    \(W\) is a \(d\)-dimensional Brownian motion with covariance matrix
      \(\bar h^{-1} \Sigma\) and
    \(\BBW(t) = \int_0^t W \otimes dW + (\bar h^{-1}E+E')t\).
\end{thm}

\begin{proof}
By Proposition~\ref{prop:gg}, it suffices to show that
$\sup_{t\le1}|W_n(t)-\tW_n(\gamma_n(t))|\to_{\mu_n}0$ and 
$\sup_{t\le 1}|\BBW_n(t)-\tBBW_n(t)\circ\gamma_n-tE'|\to_{\mu_n}0$.

First, note by Proposition~\ref{prop:flow} that
\begin{align*}
|W_n(t)-\tW_n(\gamma_n(t))|(x,u) & 
=n^{-1/2}\Big|\int_0^{nt} v_n(g_{n,s}(x,u))\,ds-\sum_{0\le j<N_n(nt)}\tilde v_n(T_n^jx)\Big| 
\\ & \le 2n^{-1/2}|h_n|_\infty|v_n|_\infty,
\end{align*}
so $\big|\sup_{t\in[0,1]}|W_n(t)-\tW_n(\gamma_n(t))|\big|_\infty\to0$.

Similarly, by Proposition~\ref{prop:flow},
\begin{align*}
 n|\BBW_n(t) & -\tBBW_n(\gamma_n(t))-n^{-1}{\SMALL\int}_0^{nt} (H_n\otimes v_n)\circ g_{n,s}\,ds|(x,u)
 \\ & \le 2|h_n|_\infty|v_n|_\infty \Big|\sum_{0\le j< N_n(t)}\tilde v_n(T_n^jx)\Big|+3|h_n|_\infty^2 |v_n|_\infty^2
 \ll  \Big|\sum_{0\le j< N_n(t)}\tilde v_n(T_n^jx)\Big|+1.
\end{align*}

By~\eqref{eq:lap}
and Theorem~\ref{thm:moments},
\begin{align*}
\Big|\sup_{t\in[0,1]}\Big|\sum_{0\le j< N_n(t)}\tilde v_n\circ T_n^j\Big|\Big|_{L^2(\Omega_n)} 
 & \ll \Big|\max_{k\le C_0n}\Big|\sum_{0\le j<k}\tilde v_n\circ f_{\Delta_n}^j\Big|\Big|_{L^2(\Lambda_n)} 
 \ll n^{1/2}\|\tilde v_n\|_\eta\ll n^{1/2},
\end{align*}
and so
 \[
\Big|\sup_{t\in[0,1]}|\BBW_n(t)  -\tBBW_n(\gamma_n(t))-n^{-1}{\SMALL\int}_0^{nt} (H_n\otimes v_n)\circ g_{n,s}\,ds|\Big|_{L^2(\Omega_n)} \to0.
\]
Also, by Corollary~\ref{cor:moment},
\[
\SMALL n^{-1}\big|\sup_{t\in[0,1]}|\int_0^{nt}(H_n\otimes v_n)\circ g_{n,s}\,ds -nt\int_{\Omega_n} H_n\otimes v_n\,d\mu_n|\big|_{L^2(\Omega_n)} \to0.
\]
Hence $\big|\sup_{t\in[0,1]}|\BBW_n(t)  -\tBBW_n(\gamma_n(t))-tE'|\big|_{L^2(\Omega_n)}\to0$ and the proof is complete.
\end{proof}

\subsection{Application to intermittent semiflows}

Let $\Lambda=[0,1]$.
Fix a family of intermittent maps $T_n:\Lambda\to\Lambda$, $n\in\N\cup\{\infty\}$, as in~\eqref{eq:LSV}
with parameters $\gamma_n\in(0,\frac12)$ and absolutely continuous invariant probability measures denoted $\tilde\mu_n$.  Suppose that
$\lim_{n\to\infty}\gamma_n=\gamma_\infty$.  
Again,
$T_n$ is a uniform family of order~$p$ for all $p\in(2,\gamma_\infty^{-1})$ and 
the absolutely continuous invariant probability measures, denoted here by $\tilde\mu_n$,
are strongly statistically stable.

Fix $\eta>0$ and 
let $h_n:\Lambda\to\R^+$ be a family of roof functions satisfying 
\mbox{$\sup_n\|h_n\|_\eta<\infty$} and $\inf_n \inf h_n>0$.  
Define the corresponding uniform family of nonuniformly expanding semiflows
$g_{n,t}:\Omega_n\to\Omega_n$ with ergodic invariant probability measures 
$\mu_n=(\tilde\mu_n\times{\rm Lebesgue})/\bar h_n$ where $\bar h_n=\int_\Lambda h_n\,d\tilde\mu_n$.

\begin{thm}
Let $v_n:\Omega_n\to\R^d$, $n\ge1$, with 
$\sup_n\|v_n\|_\eta<\infty$ and $\int_{\Omega_n}v_n\,d\mu_n=0$.
Then there is a constant $C>0$ such that
\begin{align*}
& \bigg|\sup_{t\in[0,\TT]}\Big|\int_0^t v_n\circ g_{n,s}\,ds\Big|\bigg|_{L^{2(p-1)}(\Omega_n)}\le C\TT^{1/2}, 
\\ & \bigg|\sup_{t\in[0,\TT]}\Big|\int_0^t\int_0^s (v_n\circ g_{n,r}) \otimes (v_n\circ g_{n,s}) \,dr\,ds\Big|\bigg|_{L^{p-1}(\Omega_n)}\le C\TT,
\end{align*}
for all $\TT\ge0$, $n\ge1$.
\end{thm}

\begin{proof}
This is immediate from Theorem~\ref{thm:momentflow}.
\end{proof}

\begin{thm}
Let $v_n:\Omega_n\to\R^d$, $n\in\N\cup\{\infty\}$, with 
$\sup_n\|v_n\|_\eta<\infty$ and \mbox{$\int_{\Omega_n}v_n\,d\mu_n=0$}.
Suppose that 
$\lim_{n\to\infty}\sup_{x\in\Lambda,\,u\in[0,h_\infty(x)]\cap[0,h_n(x)]}|v_n(x,u)-v_\infty(x,u)|=0$
and $\lim_{n\to\infty}|h_n-h_\infty|_\infty=0$.
\begin{itemize}
\item[(a)] 
Define 
\[
S_n = \sum_{0\le j<n}\tilde v\circ T_\infty^j,
\qquad
\BBS_n = \sum_{0\le i<j<n}
(\tilde v\circ T_\infty^i)
\otimes (\tilde v\circ T_\infty^j).
\]
where $\tilde v(x)=\int_0^{h_\infty(x)}v_\infty(x,u)\,du$.
Then the limits 
\begin{align*}  
\Sigma_\infty
 =\lim_{n\to\infty}\frac1n\int_\Lambda S_n\otimes S_n\,d\tilde\mu_\infty, \qquad
E_\infty = \lim_{n\to\infty}\frac1n \int_\Lambda \BBS_n\,d\tilde\mu_\infty.
\end{align*}
exist.
\item[(b)]
Set $E'=\int_{\Omega_\infty}H_\infty\otimes v_\infty\,d\mu_\infty$ where $H_\infty(x,u)=\int_0^u v_\infty(x,u)\,du$.
Define
\[
  W_n(t) = n^{1/2} \int_0^{ n^{-1} t } v_n \circ g_{n,s} \, ds
  , \qquad
  \BBW_n(t) = \int_0^t W_n(s) \otimes dW_n(s)
  .
\]
  Then \[(W_n, \BBW_n) \to_{\mu_n} (W, \BBW)
  \quad\text{in}\;\, D([0,1], \R^d \times \R^{d \times d}),\]
  where
    \(W\) is a \(d\)-dimensional Brownian motion with covariance matrix
      \(\bar h_\infty^{-1} \Sigma_\infty\) and
    \(\BBW(t) = \int_0^t W \otimes dW + (\bar h_\infty^{-1}E_\infty+E')t\).

\end{itemize}
\end{thm}

\begin{proof}
Part (a)  follows from Lemma~\ref{lem:En}.

To prove part~(b), we verify the hypotheses of Theorem~\ref{thm:WIPflow}.
Since $|h_n-h_\infty|_\infty\to0$ it follows from statistical stability that $\bar h_n\to \bar h_\infty$.  

Let $H_n(x,u)=\int_0^u v_n(x,s)\,ds$.
It is easy to see that $\int_0^{h_n(x)}(H_n\otimes v_n)(x,u)\,du\to\int_0^{h_\infty(x)}(H_\infty\otimes v_\infty)(x,u)\,du$ uniformly in $x$, so again by statistical stability $\int_{\Omega_n}H_n\otimes v_n\,d\mu_n\to E'$.

Finally, defining $\Sigma_n$ and $E_n$ using $T_n$, $v_n$ and $h_n$ in place of $T_\infty$, $v_\infty$ and $h_\infty$, we have that $\Sigma_n\to\Sigma_\infty$ and $E_n\to E_\infty$ by Lemma~\ref{lem:En}.
\end{proof}

\appendix

\section{Iterated WIP for martingale difference arrays}
\label{sec:mda}

In this appendix, we recast a classical iterated WIP of~\cite{JakubowskiMeminPages89,KurtzProtter91} into a form that is convenient  for ergodic stationary martingale difference arrays of the type commonly encountered in the deterministic setting.

Let $\{(\Delta_n,\cM_n,\mu_n)\}$ be a sequence of probability spaces.
Suppose that $T_n:\Delta_n\to\Delta_n$ is a sequence of measure-preserving transformations with transfer operators $L_n$
and Koopman operators $U_n$.
Suppose that $\phi_n,\,m_n:\Delta_n\to\R^d$ lie in $L^2(\Delta_n)$ and that
$\int_{\Delta_n}\phi_n\,d\mu_n=\int_{\Delta_n}m_n\,d\mu_n=0$ and $m_n \in \ker L_n$.

Define the sequence of processes 
\[
\Phi_n:\Delta_n\to D([0,\infty),\R^d), \qquad
\BBM_n:\Delta_n\to D([0,\infty),\R^{d\times d}),
\]
by
\[
\Phi_n(t)=\frac{1}{\sqrt n}\sum_{0\le j<nt} \phi_n\circ f_{\Delta_n}^j, \qquad
\BBM_n(t)=\frac1n\sum_{0\le i<j<nt}(m_n\circ f_{\Delta_n}^i)\otimes(\phi_n\circ f_{\Delta_n}^j), 
\quad t\ge0.
\]

\begin{thm} \label{thm:mda}
Suppose that:
\begin{enumerate}[label={(\alph*)}]
  \item \label{thm:mda:a}
    the family $\{|m_n|^2,\,n\ge1\}$ is uniformly integrable;
  \item \label{thm:mda:b}
    \(
      \frac{1}{\sqrt n} \max_{k\le n\TT}\Big|\sum_{j=0}^k(\phi_n-m_n)\circ f_{\Delta_n}^j\Big|\to_{\mu_n}0
    \)
      as $n\to\infty$
    for all $\TT > 0$;
  \item \label{thm:mda:c}
    there exists a constant matrix $\Sigma\in\R^{d\times d}$ such that
    for each $t>0$,
    \[
      \frac1n\sum_{j=0}^{[nt]-1}\{U_nL_n(m_n\otimes m_n)\}\circ f_{\Delta_n}^j
      \to_{\mu_n} t\Sigma
      \quad \text{ as } \ n\to\infty
      .
    \]
\end{enumerate}

Then $(\Phi_n,\BBM_n)\to_{\mu_n} (W,\BBM)$ in $D([0,\infty),\R^d\times\R^{d\times d})$ where $W$ is a $d$-dimensional Brownian motion with covariance $\Sigma$ and $\BBM(t)=\int_0^t W\otimes dW$.
\end{thm}

\begin{proof}
It suffices to prove that $(\Phi_n,\BBM_n)\to_{\mu_n} (W,\BBM)$ in $D([0,\TT],\R^d\times\R^{d\times d})$ for each fixed integer $\TT\ge1$.

Define $X_{n,j}=n^{-1/2}\phi_n\circ f_{\Delta_n}^{n\TT-j}$,
$Y_{n,j}=n^{-1/2}m_n\circ f_{\Delta_n}^{n\TT-j}$,
and
\[
X_n(t)=\sum_{1\le j\le nt}X_{n,j}, \qquad 
Y_n(t)=\sum_{1\le j\le nt}Y_{n,j}, \qquad 
\BBY_n(t)=\sum_{1\le i<j\le nt}X_{n,i}\otimes Y_{n,j}, 
\]
for $ t\in[0,\TT]$.

By the arguments in the proof of~\cite[Theorem~A.1]{KKM18},
$\{Y_{n,j};\,1\le j\le n\TT\}$ is a martingale difference array with respect to the filtration
$\cG_{n,j}=T_n^{-(n\TT-j)}\cM_n$ and
$Y_n\to_{\mu_n} W$ in $D([0,\TT],\R^d)$.
Moreover, $X_n$ is adapted (i.e.\ $X_{n,j}$ is $\cG_{n,j}$-measurable for all $j,n$) and $X_n=Y_n+Z_n$ where 
\begin{align*}
|Z_n(t)| & =\frac{1}{\sqrt n} \Big|\sum_{j=1}^{[nt]}(\phi_n-m_n)\circ f_{\Delta_n}^{n\TT-j}\Big|
 \le 
\frac{2}{\sqrt n}\max_{k\le n\TT}\Big|\sum_{j=0}^k(\phi_n-m_n)\circ f_{\Delta_n}^j \Big|,
\end{align*}
so $\sup_{t\le \TT}|Z_n(t)|\to_{\mu_n}0$ by assumption~\ref{thm:mda:b}.
It follows easily that $(X_n,Y_n)\to_{\mu_n}(W,W)$ in $D([0,\TT],\R^d\times\R^d)$.

Also
$\int_{\Delta_n}|Y_n(t)|^2\,d\mu_n=n^{-1}[nt]\int_{\Delta_n}|m_n|^2\,d\mu_n
\le \TT|m_n|_2^2$ which is bounded by assumption~\ref{thm:mda:a}, so
condition~C2.2(i) in~\cite[Theorem~2.2]{KurtzProtter91} is trivially satisfied.
Applying~\cite[Theorem~2.2]{KurtzProtter91} (or alternatively~\cite{JakubowskiMeminPages89}) we deduce that 
$(X_n,Y_n,\BBY_n)\to_{\mu_n} (W,W,\BBM)$ 
in $D([0,\TT],\R^d\times\R^d\times\R^{d\times d})$.

Next, let $\widetilde D$ denote c\`agl\`ad functions.
Adapting~\cite{KM16}, we define
$g:D([0,\TT],\R^d\times\R^d\times\R^{d\times d})\to
\widetilde D([0,\TT],\R^d\times\R^{d\times d})$,
\[
g(r,u,v)(t)=\big(r(\TT)-r(\TT-t)\,,\,\{v(\TT)-v(\TT-t)-r(\TT-t)\otimes (u(\TT)-u(\TT-t))\}^*\big),
\]
where ${}^*$ denotes matrix transpose.

We claim that 
\[
\SMALL (\Phi_n,\BBM_n)=g(X_n,Y_n,\BBY_n)+F_n\quad\text{where}\quad
\sup_{t\in[0,\TT]}|F_n(t)|\to_{\mu_n}0.
\]
Suppose that the claim is true.
By the continuous mapping theorem,
$g(X_n,Y_n,\BBY_n)\to_{\mu_n} g(W,W,\BBM)$ in $\widetilde D([0,\TT],\R^d\times\R^{d\times d})$.
Using the fact that 
the limiting process has continuous sample paths, it follows (see~\cite[Proposition~4.9]{KM16}) that
$(\Phi_n,\BBM_n)\to_{\mu_n} g(W,W,\BBM)$ in $D([0,\TT],\R^d\times\R^{d\times d})$.
By~\cite[Lemma~4.11]{KM16},
 the processes $g(W,W,\BBM)$ and $(W,\BBM)$ are equal in distribution so 
$(\Phi_n,\BBM_n)\to_{\mu_n} (W,\BBM)$ in $D([0,\TT],\R^d\times\R^{d\times d})$.

It remains to prove the claim.  Write $g=(g^1,g^2)$ where 
$g^1:D([0,\TT],\R^d\times\R^d\times\R^{d\times d})\to
\widetilde D([0,\TT],\R^d)$ and
$g^2:D([0,\TT],\R^d\times\R^d\times\R^{d\times d})\to
\widetilde D([0,\TT],\R^{d\times d})$.

First,
\begin{align*}
\Phi_n(t) & =\frac{1}{\sqrt n}\sum_{j=0}^{[nt]-1}\phi_n\circ f_{\Delta_n}^j
=\sum_{j=0}^{[nt]-1}X_{n,n\TT-j}
=\sum_{j=n\TT-[nt]+1}^{n\TT}X_{n,j} 
=\sum_{j=[n(\TT-t)]+1}^{n\TT}X_{n,j}+F_n^1(t)
\\ & =X_n(\TT)-X_n(\TT-t)+F_n^1(t)
  =g^1(X_n,Y_n,\BBY_n)(t)+F_n^1(t),
\end{align*}
where $F_n^1(t)$ is either $0$ or $-X_{n,[n(\TT-t)]+1}$.
In particular, 
\begin{align} \label{eq:F1}
|F_n^1(t)|\le n^{-1/2}\max_{i \le n\TT}|\phi_n\circ f_{\Delta_n}^i|.
\end{align}

Second,
\begin{align*}
\BBM_n(t) & =
\frac1n\sum_{0\le i<j<nt}(m_n\circ f_{\Delta_n}^i)\otimes(\phi_n\circ f_{\Delta_n}^j) 
=
\sum_{0\le i<j<nt}Y_{n,n\TT-i}\otimes X_{n,n\TT-j} 
\\ &  =
\sum_{n\TT-[nt]< j<i\le n\TT}Y_{n,i}\otimes X_{n,j} 
 =
\sum_{n\TT-[nt]< i<j\le n\TT}(X_{n,i}\otimes Y_{n,j})^* 
\\ & =\sum_{[n(\TT-t)]< i<j\le n\TT}(X_{n,i}\otimes Y_{n,j})^* +F^2_n(t)^*
 \\ & = \{\BBY_n(\TT)-\BBY_n(\TT-t)-X_n(\TT-t)\otimes(Y_n(\TT)-Y_n(\TT-t))\}^*+F^2_n(t)^*
 \\ & = g^2(X_n,Y_n,\BBY_n)(t)+F^2_n(t)^*,
\end{align*}
where 
$F^2_n(t)$ is either $0$ or
$-\sum_{[n(\TT-t)] + 1 < j \le n\TT} X_{n,[n(\TT-t)]+1} \otimes Y_{n,j}$.
In particular,
by Burkholder's inequality and assumption \ref{thm:mda:a},
\begin{align} \label{eq:F2} \nonumber
  |F_n^2(t)|
  \le n^{-1} \max_{i \le n\TT} |\phi_n\circ f_{\Delta_n}^i| \, \max_{q\le n\TT}
  \Big|\sum_{0\le j \le q}m_n\circ f_{\Delta_n}^j\Big|
  & \ll n^{-1} \max_{i \le n\TT} |\phi_n\circ f_{\Delta_n}^i|\, \TT|m_n|_2
  \\ & \ll n^{-1} \max_{i \le n\TT} |\phi_n\circ f_{\Delta_n}^i|
  .
\end{align}

By~\eqref{eq:F1} and~\eqref{eq:F2},
it remains to show that 
   $n^{-1} \max_{i \le n\TT} |\phi_n\circ f_{\Delta_n}^i|\to_{\mu_n}0$.
Note that
\[
  |\phi_n\circ f_{\Delta_n}^i|
  \leq |m_n\circ f_{\Delta_n}^i|
  + \Big| \sum_{j=0}^i (\phi_n - m_n) \circ f_{\Delta_n}^j \Big|
  + \Big| \sum_{j=0}^{i-1} (\phi_n - m_n) \circ f_{\Delta_n}^j \Big|
  ,
\]
so
\begin{align*}
  \max_{i \le n\TT} |\phi_n\circ f_{\Delta_n}^i|
  \leq \max_{i \le n\TT} |m_n\circ f_{\Delta_n}^i|
  + 2 \max_{i \le n\TT} \Big| \sum_{j=0}^i (\phi_n - m_n) \circ f_{\Delta_n}^j \Big|
  .
\end{align*}
Now for any $s>0$,
\begin{align*}
  n^{-1}\max_{j \le n\TT}|m_n\circ f_{\Delta_n}^j|^2 
  & \le s+n^{-1} \max_{j \le n\TT}(|m_n|^21_{\{n^{-1}|m_n|^2>s\}})\circ f_{\Delta_n}^j
\\& \le s+n^{-1} \sum_{j=0}^{n\TT}(|m_n|^21_{\{n^{-1}|m_n|^2>s\}})\circ f_{\Delta_n}^j.
\end{align*}
Hence
\[
    n^{-1} \big|\max_{j \le n\TT} |m_n\circ f_{\Delta_n}^j|\big|_2^2
  = n^{-1} \big|\max_{j \le n\TT} |m_n\circ f_{\Delta_n}^j|^2\big|_1 
  \le s+\TT\big||m_n|^21_{\{n^{-1}|m_n|^2>s\}}\big|_1.
\]
Since $s>0$ is arbitrary, 
it follows from assumption~\ref{thm:mda:a} that
  \(\lim_{n\to\infty}n^{-1/2}\big|\max_{j \le n\TT}|m_n\circ f_{\Delta_n}^j|\big|_2=0\).
Combining this with assumption~\ref{thm:mda:b}, 
$n^{-1/2}\max_{i\le n\TT}|\phi_n\circ f_{\Delta_n}^i|\to_{\mu_n}0$ as required.
\end{proof}

\section{Strong distributional convergence for families}
\label{sec:SDC}

In this appendix, we formulate a result on strong distributional convergence~\cite{Eagleson76, Zweimuller07} in the context of families of dynamical systems.  

Let $(\Omega_n,\mu_n)$, $n\ge1$, be a sequence of probability spaces with
measure-preserving semiflows $g_{n,t} : \Omega_n\to\Omega_n$.
Suppose that $\lambda_n$ is a sequence of probability measures on $\Omega_n$
such that
$\lambda_n\ll\mu_n$.
Define $\rho_n=d\lambda_n/d\mu_n$.

\begin{lemma} \label{lem:SDC}
  Suppose that $R_n$ is a sequence of random elements
  on $\Omega_n$ taking values in the metric space
  $(\cB,d_\cB)$ and that $R$ is a random element of $\cB$.
  Suppose moreover that 
\begin{itemize}
\item[(S1)] 
$\BIG\sup_n\int \rho_n^{1+\delta}\,d\mu_n<\infty$ for some $\delta>0$;
\item[(S2)]  
    $d_\cB(R_n \circ g_{n,t} , R_n) \to_{\mu_n}0$
as $n\to\infty$ for each $t\ge0$
(equivalently, for all $t\in[0,t_0]$ for some fixed $t_0>0$);
\item[(S3)]
    $\BIG\inf_{\TT >0} \limsup_{n \to\infty}
      \int\Big|
      \frac1\TT \int_0^\TT \rho_n \circ g_{n,t} \, dt
      \,-\, 1
      \Big| \, d\mu_n = 0$.
\end{itemize}

  Then $R_n\to_{\mu_n}R$ if and only if $R_n\to_{\lambda_n}R$.
\end{lemma}

\begin{proof}
  The proof follows~\cite[Theorem~4,1]{Gouezel_doub}.
  Let $\Lip_\cB$ denote the space of Lipschitz bounded functions
  $\psi:\cB\to\R$.
  Define $A_n(\psi,w)=\int \psi \circ R_n\, w \,d\mu_n$
  for $\psi \in \Lip_\cB$ and $w : \cB \to \R$ integrable.
  Note that $|A_n(\psi,w)| \le |\psi|_\infty |w|_1$ for all $n$.

  Now $R_n\to_{\mu_n}R$ if and only if
  $\lim_{n \to \infty}A_n(\psi,1)= \E(\psi(R))$ 
  for every $\psi\in\Lip_\cB$.
  Similarly $R_n\to_{\lambda_n}R$ if and only if
  $\lim_{n\to \infty}A_n(\psi, \rho_n)= \E(\psi(R))$
  for every $\psi\in\Lip_\cB$.
  Hence it is enough to show that
  for every $\psi\in\Lip_\cB$
  \[
    \lim_{n\to\infty} \big( A_n(\psi,\rho_n) - A_n(\psi,1) \big)
    = 0
    .
  \]

  Fix \(t \geq 0\). Since $\mu_n$ is $g_{n,t}$-invariant,
\[
    A_n(\psi, \rho_n \circ g_{n,t}) - A_n(\psi, \rho_n)
     = \int (\psi \circ R_n - \psi \circ R_n \circ g_{n,t})\,\rho_n\circ g_{n,t}\,d\mu_n.
\]
By (S1) and $g_{n,t}$-invariance,
    $\sup_{n,t}\int\rho_n^{1+\delta}\circ g_{n,t}\,d\mu_n<\infty$.
Hence by H\"older's inequality,
\[
    |A_n(\psi, \rho_n \circ g_{n,t}) - A_n(\psi, \rho_n)|
     \ll \Big(\int|\psi\circ R_n - \psi\circ R_n \circ g_{n,t}|^q\,d\mu_n\Big)^{1/q}
\]
where $q$ is the conjugate exponent to $1+\delta$.
Now $|\psi\circ R_n - \psi\circ R_n \circ g_{n,t}|\le 2|\psi|_\infty$ 
and
\[
|\psi\circ R_n - \psi\circ R_n \circ g_{n,t}|\le \Lip\,\psi\,
d_\cB(R_n , R_n \circ g_{n,t})\to_{\mu_n}0,
\]
by (S2).  Hence
  \(
    \lim_{n \to \infty}
    \big( A_n(\psi, \rho_n \circ g_{n,t}) - A_n(\psi, \rho_n) \big)
    = 0
  \)
for each $t\ge0$.
  Denote \(U_{n,\TT} = \TT^{-1} \int_0^\TT \rho_n \circ g_{n,t} \, dt\).
  Then
  \begin{equation} 
    \label{eq:AA}
    \lim_{n \to \infty} A_n (
      \psi, U_{n,\TT} - \rho_n
    )
    = 0
    ,
  \end{equation}
  for each fixed $\TT > 0$.
  Now,
  \begin{align*}
    |A_n(\psi,\rho_n)-A_n(\psi,1)|
    & \le 
    \big| A_n (\psi,\rho_n-U_{n,\TT} )\big|
    +\big| A_n(\psi,U_{n,\TT}-1) \big|
    \\ & \le
    \big|A_n(\psi,\rho_n-U_{n,\TT})\big|
    +|\psi|_\infty \int |U_{n,\TT}-1| \, d\mu_n
    .
  \end{align*}
  By~\eqref{eq:AA},
  \[
    \limsup_{n \to \infty} |A_n(\psi,\rho_n)-A_n(\psi,1)| 
    \leq |\psi|_\infty \limsup_{n\to\infty}\int |U_{n,\TT}-1| \, d\mu_n
    ,
  \]
  and the result follows from~(S3).
\end{proof}

\begin{rmk}  \label{rmk:SDC}
The discrete-time version of Lemma~\ref{lem:SDC} takes the following form.
Let $(\Lambda_n,\mu_n)$, $n\ge1$, be a sequence of probability spaces with
measure-preserving maps $T_n : \Lambda_n\to\Lambda_n$.
Suppose that $\lambda_n$ is a sequence of probability measures on $\Lambda_n$
such that
$\lambda_n\ll\mu_n$.
Define $\rho_n=d\lambda_n/d\mu_n$.
Suppose that $R_n$ is a sequence of random elements
  on $\Lambda_n$ taking values in the metric space
  $(\cB,d)$ and that $R$ is a random element of $\cB$.
We continue to assume (S1).
  Suppose moreover that
\begin{itemize}
\item[(S4)]  
    $d_\cB(R_n \circ T_n , R_n) \to_{\mu_n}0$
as $n\to\infty$;
\item[(S5)]
    $\BIG\inf_{N \geq 1} \limsup_{n \to\infty}
      \int\Big|
      \frac1N \sum_{j=0}^{N-1} \rho_n \circ T_n^j 
      \,-\, 1
      \Big| \, d\mu_n = 0$.
\end{itemize}

  Then $R_n\to_{\mu_n}R$ if and only if $R_n\to_{\lambda_n}R$.
\end{rmk}

\end{document}